\documentclass{amsart}

\usepackage{amsmath,amssymb,verbatim}
\usepackage{amsthm}
\usepackage{enumerate}
\usepackage{url}

\newtheorem{theorem}{Theorem}[section]
\newtheorem{fact}[theorem]{Fact}
\newtheorem{lemma}[theorem]{Lemma}
\newtheorem{corollary}[theorem]{Corollary}
\newtheorem{proposition}[theorem]{Proposition}

\theoremstyle{definition}
\newtheorem{definition}[theorem]{Definition}
\newtheorem{example}[theorem]{Example}
\newtheorem{remark}[theorem]{Remark}

\newcommand{\abar}{\bar{a}}
\newcommand{\bbar}{\bar{b}}
\newcommand{\cbar}{\bar{c}}
\newcommand{\dbar}{\bar{d}}
\newcommand{\ubar}{\bar{u}}
\newcommand{\vbar}{\bar{v}}
\newcommand{\xbar}{\bar{x}}
\newcommand{\ybar}{\bar{y}}
\newcommand{\zbar}{\bar{z}}

\def\seq{\subseteq}
\def\nv{\text{-}}
\def\inv{^{\text{-}1}}
\def\F{\mathbb{F}}
\def\smd{\raisebox{.4pt}{\textrm{\scriptsize{~\!$\triangle$\!~}}}}
\newcommand{\claim}{\hfill$\dashv$}

\def\Th{\operatorname{Th}}

\def\N{\mathbb{N}}
\def\R{\mathbb{R}}
\def\Z{\mathbb{Z}}
\def\Q{\mathbb{Q}}
\def\T{\mathbb{T}}

\def\cL{\mathcal{L}}
\def\cU{\mathcal{U}}
\def\cP{\mathcal{P}}
\def\cB{\mathcal{B}}

\def\cG{\mathcal{G}}

\def\cI{\mathcal{I}}

\def\tp{\operatorname{tp}}
\def\an{\operatorname{an}}

\title[Structure and regularity for VC-sets in groups]{Structure and regularity for subsets of groups with finite VC-dimension}

\date{May 25, 2020}

\author[G. Conant]{G. Conant}

\thanks{
2010 \emph{MSC}. Primary: 03C45, 03C20, 20D60, 22C05; Secondary: 22E35.\\
\indent The authors were partially supported by NSF grants: DMS-1855503 (Conant); DMS-136702, DMS-1665035, DMS-1790212 (Pillay); DMS-1855711 (Terry).}

\address{DPMMS\\
University of Cambridge\\
Cambridge CB3 0WB\\
 UK}
\email{gconant@maths.cam.ac.uk}

\author[A. Pillay]{A. Pillay}

\address{Department of Mathematics\\
University of Notre Dame\\
Notre Dame IN 46656\\
 USA}
\email{apillay@nd.edu}

\author[C. Terry]{C. Terry}

\address{Department of Mathematics\\
University of Chicago\\
Chicago IL 60637\\
 USA}
\email{caterry@uchicago.edu}

\begin{document}

\begin{abstract}
Suppose $G$ is a finite group and $A\seq G$ is such that $\{gA:g\in G\}$ has VC-dimension strictly less than $k$. We find algebraically well-structured sets in $G$ which, up to a chosen $\epsilon>0$, describe the structure of $A$ and behave regularly with respect to translates of $A$. For the subclass of groups with uniformly fixed finite exponent $r$, these algebraic objects are normal subgroups with index  bounded in terms of $k$, $r$, and $\epsilon$. For arbitrary groups, we use Bohr neighborhoods of bounded rank and width inside normal subgroups of bounded index.  Our proofs are largely model theoretic, and heavily rely on a structural analysis of compactifications of pseudofinite groups as inverse limits of Lie groups. The introduction of Bohr neighborhoods into the nonabelian setting uses model theoretic methods related to the work of Breuillard, Green, and Tao \cite{BGT} and Hrushovski \cite{HruAG} on  approximate groups, as well as a result of Alekseev, Glebski\v{\i}, and Gordon \cite{AlGlGo} on approximate homomorphisms.
\end{abstract}

\maketitle

\vspace{-15pt}

\section{Introduction and statement of results}
Szemer\'{e}di's Regularity Lemma \cite{SzemRL} is a fundamental result about graphs, which has found broad applications in graph theory, computer science, and arithmetic combinatorics. Roughly speaking, the regularity lemma  partitions large graphs into few pieces so that almost all pairs of pieces have uniform edge density. 
In 2005, Green \cite{GreenSLAG} proved the first \emph{arithmetic regularity lemma}, which uses discrete Fourier analysis  to define arithmetic notions of regularity for subsets of finite abelian groups.   For groups of the form $\F_p^n$, Green's result states that given $A\subseteq \F_p^n$, there is  $H\leq \F_p^n$ of bounded index such that $A$ is uniformly distributed in almost all cosets of $H$ (as quantified by Fourier analytic methods; see \cite[Theorem 2.1]{GreenSLAG}).  Arithmetic regularity lemmas, and their higher order analogues (see \cite{Gowers:2011et,Green:2010fk}), are now important tools in arithmetic combinatorics. From a general perspective, one can view regularity lemmas as tools for decomposing mathematical objects into ingredients that are easier to study because they are either highly structured (e.g.,  the cosets of a subgroup) or highly random (e.g., uniformly distributed).

Recently, a large body of work has developed around  strengthened regularity lemmas for classes of graphs which forbid some particular bipartite configuration. This setting is fundamental in both combinatorics and model theory, although often for very different reasons. In combinatorics, forbidden configurations can lead to significant \emph{quantitative} improvements in results about graphs, and several well-known open problems arise in this pursuit (e.g., the Erd\H{o}s-Hajnal conjecture; see 1.4 of \cite{ErHa}, and also \cite{FPS-eh}). In model theory, the focus is usually on infinite objects, and forbidden configurations are used to obtain \emph{qualitative} results about definable sets in mathematical structures. Indeed, much of modern model theory emerged from the study of mathematical structures in which every definable bipartite graph omits a finite ``half-graph" as an induced subgraph (such structures are called \emph{stable}). 

In combinatorics and model theory, the practice of forbidding finite bipartite configurations is rigorously formulated using VC-dimension. By definition, the \textbf{VC-dimension} of a \emph{bipartite} graph $(V,W;E)$ is the supremum of all $k\in\Z^+$ such that $(V,W;E)$ contains $([k],\cP([k]);\in)$ as an induced subgraph (where $[k]=\{1,\ldots,k\}$). We call $(V,W;E)$ \textbf{$k$-NIP} if it has VC-dimension at most $k-1$.\footnote{This terminology is from model theory, where a bipartite graph with infinite VC-dimension is said to have the \emph{independence property}, and so NIP stands for ``no independence property".} While this definition is based on omitting one specific bipartite graph, an illuminating exercise is that if $(V,W;E)$ omits \emph{some} finite bipartite graph $(V',W',E')$ as an induced subgraph, then $(V,W;E)$ is $k$-NIP for some   $k\leq |V'|+\lceil\log_2|W'|\rceil$.  Therefore, having finite VC-dimension is equivalent to omitting \emph{some} finite bipartite configuration. 

In \cite[Lemma 1.6]{Alon:2007gu}, Alon, Fischer, and Newman proved a strengthened regularity lemma for finite graphs of bounded VC-dimension\footnote{The VC-dimension of a graph $(V;E)$ is that of its ``bipartite double cover" $(V,V;E)$.}, in which the bound on the size of the partition is polynomial in the degree of irregularity (in contrast to Szemer\'{e}di's original work, where these bounds are necessarily tower-type \cite{Gowers:1997fh}), and the edge density in any regular pair is close to $0$ or $1$. This latter condition says that, as a bipartite graph, each regular pair is almost empty or complete, and so the normally ``random" ingredients of Szemer\'{e}di regularity are in fact highly structured. Graph regularity with bounded VC-dimension was also developed by Lov\'{a}sz and Szegedy in \cite{LovSzeg}, and similar results have been found for  hypergraphs  (e.g., \cite{FGLNP, FPS,FPS-eh}), as well as for various model theoretic settings inside NIP  (see \cite{Basu:2009cr, ChSt, ChStNIP, MaShStab}). The strongest conclusion is for \emph{stable graphs}\footnote{A bipartite graph  is called \textbf{$k$-stable} if it omits $([k],[k];\leq)$ as an induced subgraph; and a graph is \textbf{$k$-stable} if its bipartite double cover is. Note that a $k$-stable bipartite graph is $k$-NIP.}, where Malliaris and Shelah  prove the existence of regular partitions with polynomial bounds, \emph{no irregular pairs}, and the same ``$0$-$1$" behavior of edge densities in regular pairs (see \cite[Theorem 5.18]{MaShStab}).

The goal of this article is to develop arithmetic regularity for arbitrary finite groups in the context of forbidden bipartite configurations, as quantified by VC-dimension. In analogy to the case of graphs, we show that by forbidding finite bipartite configurations, one obtains a strengthened version of arithmetic regularity in which the normally random ingredients are instead highly structured. Moreover, our proof methods deepen the connection between model theory and arithmetic combinatorics, in that we use pseudofinite methods to extend combinatorial results for finite abelian groups to the nonabelian setting. This is in the same vein as Hrushovski's \cite{HruAG} celebrated work on approximate groups, and the subsequent structure theory proved by Breuillard, Green, and Tao \cite{BGT}. We will use similar techniques  in order to formulate arithmetic regularity in nonabelian groups using \emph{Bohr neighborhoods}, which are fundamental objects from arithmetic combinatorics in abelian groups. Finally, our results show that arithmetic regularity for NIP sets in finite groups coincides with a certain model theoretic phenomenon called ``compact domination". This notion was first isolated by Hrushovski, Peterzil, and the second author \cite{HPP} in their proof of the so-called ``Pillay conjectures" for groups definable in o-minimal theories, and later played an important role in the study of definably amenable groups definable in NIP theories \cite{ChSi,HPP, HP,HPS}. 

Before stating the main results of this paper, we briefly recall previous work on stable arithmetic regularity, as it provides a template for  the ``structure and regularity" statements we will obtain in the NIP setting. Given a group $G$ and a subset $A\seq G$, we define the bipartite graph $\Gamma_G(A)=(V,W;E)$ where $V=W=G$ and $E=\{(x,y)\in G^2:yx\in A\}$.\footnote{This is a bipartite analogue of the ``Cayley sum-graph of $A$ in $G$", as defined in \cite{TeWo} for abelian groups. The reason we use $yx\in A$ rather than $xy\in A$ is due to the model theoretic preference for ``left-invariant" formulas (see Section \ref{sec:NIPsetsG} for details).} Given $k\geq 1$, we say that a subset $A$ of a group $G$ is \textbf{$k$-NIP} (respectively, \textbf{$k$-stable}) if $\Gamma_G(A)$ is $k$-NIP (respectively, $k$-stable), as defined above. In \cite{TeWo}, the third author and Wolf  developed arithmetic regularity for $k$-stable subsets of $\F_p^n$. They proved that such sets satisfy a strengthened version of Green's arithmetic regularity lemma above, in which there is an efficient bound on the index of $H$ and $A$ is uniformly distributed in \emph{all} cosets of $H$. They also show that a $k$-stable subset of $\F_p^n$ is approximately a union of cosets of a subgroup of small index, which is an arithmetic analogue of ``$0$-$1$ density" in regular pairs. In \cite[Theorem 1.2]{CPT}, we generalized and strengthened the results from \cite{TeWo} on $\F_p^n$ to the setting of arbitrary finite groups, but without explicit bounds.\footnote{Quantitative results for stable sets in finite abelian groups were later proved by the third author and Wolf (see \cite[Theorem 4]{TeWo2}), and more recently for arbitrary finite groups by the first author (see Theorems 1.3 and 1.4 of \cite{CoQSAR}).} 

\begin{theorem}\textnormal{\cite{CPT}}\label{thm:CPT1}
For any $k\geq 1$ and $\epsilon>0$, there is $n=n(k,\epsilon)$ such that the following holds. Suppose $G$ is a finite group and $A\seq G$ is $k$-stable. Then there is a normal subgroup $H\leq G$, of index at most $n$, satisfying the following properties.
\begin{enumerate}[$(i)$]
\item \textnormal{(structure)} There is a set $D\seq G$, which is a union of cosets of $H$, such that
\[
|A\smd D|<\epsilon|H|.
\]
\item \textnormal{(regularity)} For any $g\in G$, either $|gH\cap A|<\epsilon|H|$ or $|gH\backslash A|<\epsilon |H|$.
\end{enumerate}
Moreover, $H$ is in the Boolean algebra generated by $\{gAh:g,h\in G\}$.
\end{theorem}

Our first result on arithmetic regularity in the setting of bounded VC-dimension is for $k$-NIP subsets of finite groups with \emph{uniformly bounded exponent}.

\newtheorem*{thm:regexp}{Theorem \ref{thm:regexp}}
\begin{thm:regexp}
\emph{
 For any $k,r\geq 1$ and $\epsilon>0$, there is $n=n(k,r,\epsilon)$ such that the following holds. Suppose $G$ is a finite group of exponent $r$, and $A\seq G$ is $k$-NIP. Then there are
\begin{enumerate}[\hspace{5pt}$\ast$]
\item  a normal subgroup $H\leq G$ of index at most $n$, and
\item   a set $Z\seq G$, which is a union of cosets of $H$ with $|Z|<\epsilon|G|$,
\end{enumerate}
satisfying the following properties.
\begin{enumerate}[$(i)$]
\item \textnormal{(structure)} There is a set $D\seq G$, which is a union of cosets of $H$, such that
\[
|(A\backslash Z)\smd D|<\epsilon|H|.
\]
\item \textnormal{(regularity)} For any $g\in G\backslash Z$, either $|gH\cap A|<\epsilon|H|$ or $|gH\backslash A|<\epsilon|H|$.
\end{enumerate}
Moreover, $H$ is in the Boolean algebra generated by $\{gAh:g,h\in G\}$.
}
\end{thm:regexp}

Thus the behavior of NIP sets in bounded exponent groups is almost identical to that of stable sets in arbitrary finite groups, where the only difference is the error set $Z$. This reflects similar behavior in graph regularity, where the main difference between the stable and NIP cases is the need for irregular pairs.  Theorem \ref{thm:regexp} also qualitatively generalizes and strengthens a recent quantitative result of Alon, Fox, and Zhao \cite[Theorem 1.1]{AFZ} on $k$-NIP subsets of finite \emph{abelian} groups of uniformly bounded exponent.
A  version of Theorem \ref{thm:regexp} with polynomial bounds (in $\epsilon\inv$), but weaker qualitative ingredients, is conjectured in \cite{AFZ}.\footnote{This weaker version of Theorem \ref{thm:regexp} was later proved by the first author with a bound of the form $n=\exp(c_{k,r}\epsilon^{\nv k})$, where $c_{k,r}$ is an ineffective constant (see \cite[Theorem 1.6]{CoBogo}).}

We then turn to $k$-NIP sets in arbitrary finite groups. In this case, one cannot expect a statement involving only subgroups, as in Theorem \ref{thm:regexp}. Indeed, as noted in \cite[Section 5]{AFZ}, if $p\geq 3$ is prime and $A=\{1,2,\ldots,\lfloor \frac{p}{2}\rfloor\}$, then $A$ is $4$-NIP as a subset of $\Z/p\Z$, but $A$ cannot be approximated as in Theorem \ref{thm:regexp} for arbitrarily small $\epsilon$. This example illustrates a common obstacle faced in arithmetic combinatorics when working in abelian groups with very few subgroups. In situations like this, one often works instead with certain well-structured subsets of groups called  \emph{Bohr neighborhoods}. Bohr neighborhoods in cyclic groups were used in Bourgain's improvement of Roth's Theorem  \cite[(0.11)]{BourgTAP}, and also form the basis of Green's arithmetic regularity lemma for finite abelian groups \cite[Theorem 5.2]{GreenSLAG}. Results related to ours involving Bohr neighborhoods, with quantitative bounds but for abelian groups, were independently obtained by Sisask  \cite[Theorem 1.4]{Sisask}.

Given a group $H$, an integer $r\geq 0$, and some $\delta>0$, we define a \textbf{$(\delta,r)$-Bohr neighborhood in $H$} to be a set of the form $B^r_{\tau,\delta}:=\{x\in H:d(\tau(x),0)<\delta\}$, where $\tau\colon H\to\T^r$ is a group homomorphism and $d$ is a fixed invariant metric on the $r$-dimensional torus $\T^r$ (see Remark \ref{rem:Liemetric} and Definition \ref{def:BohrT}). Our main structure and regularity result for NIP sets in finite groups is as follows.

\newtheorem*{thm:mainNIP}{Theorem \ref{thm:mainNIP}}
\begin{thm:mainNIP}
\emph{
For any $k\geq 1$ and $\epsilon>0$ there is $n=n(k,\epsilon)$ such that the following holds. Suppose $G$ is a finite group and $A\seq G$ is $k$-NIP. Then there are
\begin{enumerate}[\hspace{5pt}$\ast$]
\item  a normal subgroup $H\leq G$ of index $m\leq n$, 
\item  a $(\delta,r)$-Bohr neighborhood $B$ in $H$, where $0\leq r\leq n$ and $\frac{1}{n}\leq\delta\leq 1$, and
\item a subset $Z\seq G$, with $|Z|<\epsilon|G|$,
\end{enumerate}
satisfying the following properties. 
\begin{enumerate}[$(i)$]
\item \textnormal{(structure)} There is a set $D\seq G$, which is a union of at most $m(\frac{2}{\delta})^r$ translates of $B$, such that 
\[
|(A\smd D)\backslash Z|< \epsilon|B|.
\]
\item \textnormal{(regularity)} For any $g\in G\backslash Z$, either $|gB\cap A|<\epsilon|B|$ or $|gB\backslash A|< \epsilon|B|$.
\end{enumerate}
Moreover, $H$ and $Z$ are in the Boolean algebra generated by $\{gAh:g,h\in G\}$, and if $G$ is abelian then we may assume $H=G$.
}
\end{thm:mainNIP}

In order to prove Theorems \ref{thm:regexp} and \ref{thm:mainNIP}, we will first prove companion theorems for these results involving definable sets in infinite pseudofinite groups (Theorems \ref{thm:UPprof} and \ref{thm:UPgen}, respectively). We then  prove the theorems about finite groups by taking ultraproducts of counterexamples in order to obtain  infinite pseudofinite groups contradicting the companion theorems. To prove the companion theorems, we work with a saturated pseudofinite group $G$, and an invariant NIP formula $\theta(x;\ybar)$ (see Definitions \ref{def:localdefns} and \ref{def:NIP}). In \cite{CPpfNIP}, the first two authors proved ``generic compact domination" for the quotient group $G/G^{00}_{\theta^r}$, where $\theta^r(x;\ybar,u):=\theta(x\cdot u;\ybar)$ and $G^{00}_{\theta^r}$ is the intersection of all $\theta^r$-type-definable bounded-index subgroups of $G$ (see Definition \ref{def:G00dr}).  In this case, $G/G^{00}_{\theta^r}$ is a compact Hausdorff group, and generic compact domination roughly states that if $A\seq G$ is $\theta^r$-definable, then the set of cosets of $G^{00}_{\theta^r}$, which intersect both $A$ and $G\backslash A$ in ``large" sets with respect to the pseudofinite counting measure, has Haar measure $0$ (see Theorem \ref{thm:G00CP}). This is essentially a regularity statement for $A$ with respect to the subgroup $G^{00}_{\theta^r}$, which is rather remarkable as generic compact domination originated in \cite{HPP} toward proving conjectures of the second author on the Lie structure of groups definable in o-minimal theories \cite[Conjecture 1.1]{PilCLG}. 

For $G$ and $\theta(x;\ybar)$ as above, the regularity provided by generic compact domination for $\theta^r$-definable sets in $G$ cannot be transferred directly to finite groups, as the statement depends entirely on type-definable data (such as $G^{00}_{\theta^r}$). Thus, much of the work in this paper focuses on obtaining definable approximations to $G^{00}_{\theta^r}$ and the other objects involved in generic compact domination. We first investigate the situation when $G/G^{00}_{\theta^r}$ is a profinite group, in which case $G^{00}_{\theta^r}$ can be approximated by definable finite-index subgroups of $G$. Using this, we prove Theorem \ref{thm:UPprof} (the pseudofinite companion to Theorem \ref{thm:regexp} above). The connection to Theorem \ref{thm:regexp} is that if $G$ is elementarily equivalent to an ultraproduct of groups of uniformly bounded exponent, then $G/G^{00}_{\theta^r}$ is a compact Hausdorff group of finite exponent, hence is profinite (see Fact \ref{fact:compactG}$(d)$).

When $G/G^{00}_{\theta^r}$ is not profinite,  there are not enough definable finite-index subgroups available to describe $G^{00}_{\theta^r}$, and it is for this reason that we turn to Bohr neighborhoods. This is somewhat surprising, as Bohr neighborhoods are fundamentally linked to abelian groups, and we do not make any assumptions of commutativity. However,  by a general result of the second author on ``definable" compactifications of pseudofinite groups, we in fact have that the \emph{connected component} of $G/G^{00}_{\theta^r}$ is abelian (see Theorem \ref{thm:comm}). It is at this point that we see the beautiful partnership between pseudofinite groups and NIP formulas. Specifically, we have generic compact domination of $\theta^r$-definable sets by the abelian-by-profinite group $G/G^{00}_{\theta^r}$, which allows us to analyze $\theta^r$-definable sets in $G$ using Bohr neighborhoods in definable finite-index subgroups. In order to obtain a statement involving only \emph{definable} objects, we use approximate homomorphisms to formulate a notion of approximate Bohr neighborhoods. This leads to Theorem \ref{thm:UPgen} (the pseudofinite companion of Theorem \ref{thm:mainNIP} above). We then apply a result of Alekseev, Glebski\v{\i}, and Gordon \cite[Theorem 5.13]{AlGlGo} on approximate homomorphisms  to find actual Bohr neighborhoods inside approximate Bohr neighborhoods and, ultimately, prove Theorem \ref{thm:mainNIP}.

In Section \ref{sec:distal}, we prove similar results for \emph{fsg} groups definable in \emph{distal} NIP theories (see Theorems \ref{thm:distalgen} and \ref{thm:distalprof}). This follows our theme, as such groups satisfy a strong form of compact domination (see Lemma \ref{lem:CD}). In Section \ref{sec:padic}, we discuss compact $p$-adic analytic groups as one concrete example of the distal \emph{fsg} setting. 

\subsection*{Acknowledgements}  Much of the work in this paper was carried out during the 2018 Model Theory, Combinatorics and Valued Fields trimester program at Institut Henri Poincar\'{e}. We thank IHP for their hospitality. We also thank Julia Wolf for  helpful conversations, and the anonymous referee for their careful reading and for providing many valuable revisions, corrections, and suggestions.

\section{Preliminaries}\label{sec:pre}

\subsection{First-order structures and definability} We start by establishing the setting involving first-order structures. See \cite{Mabook} for an introduction to first-order logic and model theory. 

Let $\cL$ be a first-order language. Following model-theoretic convention, we will say that an $\cL$-structure $M^*$ is \textbf{sufficiently saturated} if $M^*$ is $\kappa$-saturated and strongly $\kappa$-homogeneous for some large (e.g., strongly inaccessible) cardinal $\kappa$. Let $M^*$ be a fixed sufficiently saturated $\cL$-structure. We say that a set is \textbf{bounded} if its cardinality is strictly less than the saturation cardinal $\kappa$ of $M^*$.\footnote{This is not to be confused with later uses of the phrase ``uniformly bounded" in the context of theorems about finite groups.} When $\cL$ and $M^*$ are fixed, we will refer to an $\cL$-formula with parameters from $M^*$ as simply a \textbf{formula}. We also call an elementary substructure $M\prec M^*$ \textbf{small} if its universe is bounded. As usual, a subset $X\seq (M^*)^n$ is \textbf{definable} if there is a formula $\phi(x_1,\ldots,x_n)$ such that $X=\phi(M^*):=\{\abar\in (M^*)^n:M^*\models\phi(\abar)\}$; and $X\seq (M^*)^n$ is \textbf{type-definable} if it is an intersection of a bounded number of definable subsets of $(M^*)^n$. We will further say that $X\seq (M^*)^n$ is \textbf{countably-definable} if it is an intersection of countably many definable subsets of $(M^*)^n$. 

Many of our results will operate in a ``local" model-theoretic setting, which is to say that we focus on a single formula $\theta(x;\ybar)$ whose free variables  are partitioned into a singleton $x$ and a finite tuple $\ybar$.

\begin{definition}\label{def:localdefns}
Let $\theta(x;\ybar)$ be a formula.
\begin{enumerate}
\item A \textbf{$\langle \theta\rangle$-formula} is a formula $\phi(x;\ybar_1,\ldots,\ybar_n)$ obtained as a Boolean combination of $\theta(x;\ybar_1),\ldots,\theta(x;\ybar_n)$ for some $n\geq 1$ and variables $\ybar_1,\ldots,\ybar_n$.
\item An \textbf{instance of $\theta(x;\ybar)$} is a formula of the form $\theta(x;\bbar)$ for some $\bbar\in (M^*)^{|\ybar|}$. 
\item A \textbf{$\theta$-formula} $\phi(x)$ is an instance of a $\langle \theta\rangle$-formula $\phi(x;\ybar_1,\ldots,\ybar_n)$.
\item A set $X\seq M^*$ is \textbf{$\theta$-definable} if $X=\phi(M^*)$ for some $\theta$-formula $\phi(x)$. 
\item A set $X\seq M^*$ is \textbf{$\theta$-type-definable} (respectively, \textbf{$\theta$-countably-definable}) if it is an intersection of a bounded (respectively, countable) number of $\theta$-definable subsets of $M^*$. 
\end{enumerate}
\end{definition}

Note that if $\theta(x;\ybar)$ is a formula, and $X\seq M^*$ is a set that is both countably-definable and $\theta$-type-definable, then $X$ is $\theta$-countably-definable by an easy saturation argument.

We end this subsection with some discussion on our conventions regarding definable groups. In general, we will work with a sufficiently saturated structure $M^*$ and consider a group $G$ definable in $M^*$. (For example, this will be the setting for the results in Sections \ref{sec:distal} and \ref{sec:padic}.) Thus most of the preliminaries below will take place in this context. On the other hand, the  results in Sections \ref{sec:prof} and \ref{sec:gen} are in the ``local setting" discussed above, and thus we will work in the situation where $\cL$ is an expansion of the group language\footnote{In fact, the group language expanded by a single unary relation symbol $A$ suffices for our main results on finite groups (Theorems \ref{thm:regexp} and \ref{thm:mainNIP}).} and $M^*$ expands a group. So in this case, we use $G$ in place of $M^*$, and we just say that $G$ is  a \emph{sufficiently saturated expansion of a group}.

In some situations, we may start by working in the setting of groups definable in $M^*$ and then, for certain results, move to the latter setting where $M^*$ is itself an expansion of a group. To clarify this, we adopt the following convention. Given a sufficiently saturated structure $M^*$ and a group $G$ definable in $M^*$, we will write $G=M^*$ to signify the additional assumption that $\cL$ contains the language of groups, $M^*$ is an expansion of a group, and $G$ is defined by $x=x$. In other words, this assumption moves us to the setting where we work with a sufficiently saturated expansion of a group (as discussed above).  This will be done primarily for the sake of convenience, and in order to avoid introducing further terminology and notation. 
 
 \begin{definition}
 Let $G$ be a sufficiently saturated expansion of a group. A formula $\theta(x;\ybar)$ is \textbf{(left) invariant} if, for any $a\in G$ and $\bbar\in G^{|\ybar|}$, there is $\cbar\in G^{|\ybar|}$ such that $\theta(a\cdot x;\bbar)$ and $\theta(x;\cbar)$ define the same subset of $G$.
 \end{definition}

The typical example of an invariant formula is something of the form $\theta(x;y):=\phi(y \cdot x)$, where $\phi(x)$ is a formula in one variable. Note also that if $\theta(x;\ybar)$ is an invariant formula then any $\langle \theta\rangle$-formula is invariant as well.

\subsection{Compact quotients}\label{sec:preCQ}

By convention, when we say that a topological space is  \textbf{compact}, we mean \emph{compact and Hausdorff}. We will frequently use the fact that a compact space is second-countable if and only if it is separable and metrizable (e.g., by \emph{Urysohn's Metrization Theorem}). Given a topological group $K$, we let $K^0$ denote the connected component of the identity in $K$, which is a closed normal subgroup of $K$. 
A topological group  is called \textbf{profinite} if it is a projective limit of finite groups. By a \textbf{Lie group}, we mean a finite-dimensional real Lie group. For $n\in\N$, let $\T^n$ denote the $n$-dimensional torus, where $\T=\R/\Z$ (so $\T^0$ is the trivial group). We view $\T^n$ as a compact topological group (with the product of the quotient topology on $\R/\Z$). We use $\cong$ for isomorphism of topological groups.

The structure theory for compact groups as projective limits of compact Lie groups will play a significant role in this paper. We refer the reader to \cite[Section 1.1]{RZbook} for details on projective limits of topological spaces. Each projective limit below will be indexed by a \textbf{join semilattice} (\emph{jsl}), i.e., a poset $I=(I,\leq)$ such that any finite set $F\seq I$ has a least upper bound (denoted $\sup F$). When we refer to $\N$ as a \emph{jsl}, we always use the usual ordering. 

The following are several important facts about compact groups that we will need at various points throughout the paper. 

\begin{fact}\label{fact:compactG}$~$
\begin{enumerate}[$(a)$]
\item If $K$ and $L$ are compact groups, and $\pi\colon K\to L$ is a surjective continuous homomorphism, then $\pi(K^0)=L^0$.
\item A  compact group $K$ is profinite if and only if $K^0$ is trivial.  \item A continuous homomorphic image of a profinite group is profinite.
\item Any compact torsion group is profinite.  
\item Any compact second-countable group admits a bi-invariant compatible metric. 
\item Any compact group $K$ admits a unique left-invariant regular Borel probability measure $\eta_K$ (called the \textbf{normalized Haar measure} on $K$).\footnote{In fact, $\eta_K$ is also right-invariant and is the unique right-invariant regular Borel probability measure on $K$.}
\item If $K$ is a compact abelian Lie group, then $K\cong\T^n\times F$ for some finite group $F$ and some $n\in\N$.
\item If $K$ is a compact group then there is a jsl $I$ and a projective system $(L_i)_{i\in I}$ of compact Lie groups such that $K\cong\varprojlim L_i$. Moreover, the projection maps  are surjective, and if $K$ is second-countable then we may assume $I=\N$.
\end{enumerate}
\end{fact} 
\begin{proof}
Part $(a)$. This is a straightforward consequence of the Open Mapping Theorem for compact groups \cite[p. 704]{HofMo3}. See also \cite[Theorem 7.12]{HeRobook} for a more general result in the locally compact setting.


Part $(b)$. This follows from the fact that a topological group is profinite if and only if it is compact and totally disconnected (see \cite[Theorem 2.1.3]{RZbook}). 

Part $(c)$. This follows from parts $(a)$ and $(b)$. 

Part $(d)$. See \cite[Theorem 4.5]{Ilt} (where torsion groups are called \emph{periodic}).

Part $(e)$. See \cite[Corollary A4.19]{HofMo3}. 

Part $(f)$. See \cite[Theorem 2.8]{HofMo3}.

Part $(g)$. See \cite[Proposition 2.42]{HofMo3}.

Part $(h)$. This is part of the well-known \emph{Peter-Weyl Theorems}. See \cite[Corollary 2.43]{HofMo3} for the first statement and surjectivity of the projection maps. The \emph{jsl} in question is the collection $I$ of kernels of continuous homomorphisms from $K$ to unitary groups under reverse inclusion (see \cite[Corollary 2.36]{HofMo3} and its proof). In particular, any open neighborhood of the identity in $K$ contains some kernel in $I$ (see \cite[Exercise E9.1]{HofMo3}), which yields the final claim on second-countability.
\end{proof}

We now return to model theory. Let $\cL$ be a first-order language. 
Given a sufficiently saturated $\cL$-structure $M^*$, a type-definable set $X\seq (M^*)^n$, and a complete $n$-type $p\in S_n(M^*)$, we say $p$ \textbf{concentrates} on $X$, written $p\models X$, if $p$ contains a bounded set $\{\phi_i(\xbar)\}_{i\in I}$ of formulas such that $X=\bigcap_{i\in I}\phi_i(M^*)$ (note that, by saturation and since $p$ is a complete type, we have $p\models X$ if and only if $\phi(\xbar)\in p$ for \emph{any} formula $\phi(\xbar)$ such that $X\seq \phi(M^*)$).

\begin{fact}\label{fact:QLT}
Let $G$ be a group definable in a sufficiently saturated $\cL$-structure $M^*$.
Suppose $\Gamma\leq G$ is a type-definable normal subgroup of bounded index, and let $\pi\colon G\to G/\Gamma$ be the canonical homomorphism.
\begin{enumerate}[$(a)$]
\item $G/\Gamma$ is a compact topological group, where a set $X\seq G/\Gamma$ is \textbf{closed} if and only if $\pi\inv(X)$ is type-definable.
\item If $X\seq G$ is definable then the set $\{a\Gamma\in G/\Gamma:a\Gamma\seq X\}$ is open in $G/\Gamma$.
\item $G/\Gamma$ is second-countable if and only if $\Gamma$ is countably-definable. 
\item Suppose $G=M^*$ and $\Gamma$ is $\theta$-type-definable for some invariant formula $\theta(x;\ybar)$. Then $X\seq G/\Gamma$ is closed if and only if $\pi\inv (X)$ is $\theta$-type-definable.
\item Suppose $\mu$ is a left-invariant finitely additive probability measure on the Boolean algebra $\cB$ of definable subsets of $G$, and $X$ is a closed subset of $G/\Gamma$.  Then $\eta_{G/\Gamma}(X)=\inf\{\mu(Z):\text{$Z\in\cB$ and $\pi\inv(X)\seq Z$}\}$. 
\end{enumerate}
\end{fact}
\begin{proof}
Part $(a)$. See \cite[Lemma 2.7]{PilCLG}.

Part $(b)$.  See \cite[Remark 2.4]{PilCLG}.

Part $(c)$. See \cite[Fact 1.3]{KrNe} for the right-to-left direction. Conversely, suppose $G/\Gamma$ is second-countable and let $\{U_n:n\in\N\}$ be a neighborhood basis of the identity in $G/\Gamma$. For any $n\in\N$, $\pi\inv(U_n)$ is co-type-definable and contains $\Gamma$. By saturation, there is some definable set $D_n\seq G$ such that $\Gamma\seq D_n\seq\pi\inv(U_n)$.  It follows that $\Gamma=\bigcap_{n=0}^\infty D_n$.

Part $(d)$. See \cite[Corollary 4.2]{CPpfNIP}.

Part $(e)$. Let $S_G(M^*)$ be the space of types concentrating on $G$. Recall that $S_G(M^*)$ is a Stone space, with a basis of clopen sets  $[Z]:=\{p\in S_G(M^*):Z\in p\}$ where $Z\in\cB$. Note that $G$ acts on $S_G(M^*)$ by left multiplication. Moreover, there is a unique left-invariant regular Borel probability measure $\widetilde{\mu}$ on $S_G(M^*)$ such that $\mu(Z)=\widetilde{\mu}([Z])$ for any $Z\in\cB$ (see, e.g., \cite[Section 7.1]{Sibook}). In particular, if $C\seq S_G(M^*)$ is closed then $\widetilde{\mu}(C)=\inf\{\mu(Z):Z\in\cB\text{ and }C\seq [Z]\}$. Now define the function $f\colon S_G(M^*)\to G/\Gamma$ where $f(p)$ is the unique coset $a\Gamma$ such that $p\models a\Gamma$. It is straightforward to show that $f\inv(X)=\{p\in S_G(M^*):p\models \pi\inv(X)\}$ for any closed $X\seq G/\Gamma$. It follows that $f$ is continuous and that  $\widetilde{\mu}(f\inv(X))=\inf\{\mu(Z):Z\in\cB\text{ and }\pi\inv(X)\seq Z\}$ for any closed $X\seq G/\Gamma$.  Finally,  one checks that $\widetilde{\mu}\circ f\inv$ is a left-invariant regular Borel probability measure on $G/\Gamma$, and thus must be $\eta_{G/\Gamma}$ by Fact \ref{fact:compactG}$(f)$. 
\end{proof}

The topology defined on $G/\Gamma$ in the previous fact is called the \emph{logic topology}, and the canonical  homomorphism from $G$ to $G/\Gamma$ connects topological structure in $G/\Gamma$ to definable sets in $G$. This is a special case of the following notion.

\begin{definition}
Suppose $M^*$ is a sufficiently saturated  $\cL$-structure, $X\seq M^*$ is definable, $Y$ is a compact space, and $f\colon X\to Y$ is a function. Then $f$ is \textbf{definable} if $f\inv(C)$ is type-definable for any closed $C\seq Y$. If $X$ is $\theta$-definable and each $f\inv(C)$ is $\theta$-type-definable, for some fixed formula $\theta(x;\ybar)$, then we say $f$ is \textbf{$\theta$-definable}.
\end{definition}

\begin{remark}\label{rem:defmap}
Suppose $f\colon X\to Y$ is definable (as above).
\begin{enumerate}[$(a)$]
\item Fix $C\seq U\seq Y$ with $C$ closed and $U$ open. Then, by saturation of $M^*$,  there is a definable set $D\seq X$ such that $f\inv(C)\seq D\seq f\inv(U)$. Thus if $C\seq Y$ is clopen then $f\inv(C)$ is definable. 
\item If $f(X)$ is finite then the fibers of $f$ partition  $X$ into finitely many type-definable sets which, by saturation, implies that each fiber of $f$ is definable. 
\end{enumerate}
Suppose further that $X$ and $f$ are $\theta$-definable for some fixed formula $\theta(x;\ybar)$. Then in part $(a)$ we may assume $D$ is $\theta$-definable, and so if $C\seq Y$ is clopen then $f\inv(C)$ is $\theta$-definable. Moreover, in part $(b)$, we have that each fiber of $f$ is $\theta$-definable.
\end{remark}

Recall that a first-order structure is \textbf{pseudofinite} if it is elementarily equivalent to an ultraproduct of finite structures. We now recall a result of the second author on definable compactifications of pseudofinite groups,  which will be the key ingredient that  allows us to introduce Bohr neighborhoods into the setting of possibly nonabelian finite groups.

\begin{theorem}\label{thm:comm}\textnormal{\cite{PiRCP}}
Suppose $G$ is a sufficiently saturated pseudofinite expansion of a group, and $\Gamma\leq G$ is normal and type-definable of bounded index. Then $(G/\Gamma)^0$ is abelian.
\end{theorem}
\begin{proof}
Suppose $M\prec G$ is a small substructure such that $\Gamma$ is type-definable over $M$. Then $G^{00}_M\leq \Gamma$, where $G^{00}_M$ denotes the intersection of \emph{all} bounded-index subgroups of $G$ that are type-definable over $M$. By \cite[Theorem 2.2]{PiRCP}, $(G/G^{00}_M)^0$ is abelian, and so $(G/\Gamma)^0$ is abelian by Fact \ref{fact:compactG}$(a)$ and since  the canonical homomorphism from $G/G^{00}_M$ to $G/\Gamma$ is surjective and continuous.
\end{proof}

\begin{remark}
The proof of \cite[Theorem 2.2]{PiRCP} uses the classification of approximate groups due to Breuillard, Green, and Tao \cite{BGT}. This theorem was later generalized by Nikolov, Schneider, and Thom \cite[Theorem 8]{NST}, who showed that the connected component of \emph{any} compactification of an (abstract) pseudofinite group is abelian. Their proof uses the classification of finite simple groups.
\end{remark}

We are now ready to collect all of the previous facts and prove the main result of this subsection. In particular, let $G$ be a group definable in a sufficiently saturated structure $M^*$, and let $\Gamma\leq G$ be type-definable and normal of bounded index. Since $G/\Gamma$ is a compact group, and the topology on $G/\Gamma$ is controlled by type-definable objects in $G$, we can analyze $\Gamma$ using the structure theory for compact Lie groups. The next lemma describes the main ingredients of this analysis. Given a poset $I$, we say that a net $(X_i)_{i\in I}$ of subsets of some fixed set $X$ is \textbf{decreasing} if $X_j\seq X_i$ for any $i,j\in I$ such that $i\leq j$.

\begin{lemma}\label{lem:Lie}
Let $G$ be a group definable in a sufficiently saturated $\cL$-structure $M^*$. Suppose $\Gamma\leq G$ is type-definable and normal of bounded index. Then there is a bounded jsl $I$, a decreasing net $(\Gamma_i)_{i\in I}$ of countably-definable bounded-index normal subgroups of $G$, and decreasing net $(H_i)_{i\in I}$ of definable finite-index normal subgroups of $G$, such that the following properties are satisfied.
\begin{enumerate}[$(i)$]
\item$\Gamma=\bigcap_{i\in I}\Gamma_i$, and $\Gamma_i\leq H_i$ for all $i\in I$.
\item For all $i\in I$, $G/\Gamma_i$ is a compact Lie group and $H_i/\Gamma_i=(G/\Gamma_i)^0$. 
\item $G/\Gamma\cong \varprojlim G/\Gamma_i$ and $(G/\Gamma)^0\cong\varprojlim H_i/\Gamma_i$.\footnote{Here we view $(G/\Gamma_i)_{i\in I}$ and $(H_i/\Gamma_i)_{i\in I}$ as projective systems using $(i)$ and the assumption that $(H_i)_{i\in I}$ and $(\Gamma_i)_{i\in I}$ are decreasing.}
\end{enumerate}
Moreover:
\begin{enumerate}[$(a)$]
\item If $\Gamma$ is countably-definable then we may assume $I=\N$.
\item If $G=M^*$ and $\Gamma$ is $\theta$-type-definable for some invariant formula $\theta(x;\ybar)$ then, for all $i\in I$, the quotient map $G\to G/\Gamma_i$ is $\theta$-definable, and so we may assume that $H_i$ is $\theta$-definable and  $\Gamma_i$ is $\theta$-type-definable.
\item If $G/\Gamma$ is profinite then we may assume $H_i=\Gamma_i$ for all $i\in I$.
\item If $G=M^*$ and $G$ is pseudofinite, then we may assume that for all $i\in I$, there is some $n_i\in\N$ such that $H_i/\Gamma_i\cong \T^{n_i}$.
\item If $G$ is abelian then we may assume that, for all $i\in I$, there is some $n_i\in\N$ and some finite group $F_i$ such that $G/\Gamma_i\cong \T^{n_i}\times F_i$ (and so $H_i/\Gamma_i\cong \T^{n_i}$). 
\end{enumerate} 
\end{lemma}
\begin{proof}
By Fact \ref{fact:compactG}$(h)$, there is a projective system $(L_i)_{i\in I}$ of compact Lie groups such that $G/\Gamma\cong\varprojlim L_i$ and the projection maps $f_i\colon G/\Gamma\to L_i$ are surjective. Since $\Gamma$ has bounded index in $G$, we may assume $I$ is bounded.
For $i\in I$, set $\Gamma_i=\ker(f_i\circ\pi)$, where $\pi\colon G\to G/\Gamma$ is the canonical homomorphism. Then each $\Gamma_i$ is a type-definable bounded-index normal subgroup of $G$, and $\Gamma=\bigcap_{i\in I}\Gamma_i$. Moreover, $G/\Gamma_i\cong L_i$ for all $i\in I$, and so $G/\Gamma\cong\varprojlim G/\Gamma_i$ with canonical (surjective) projection maps $\tau_i\colon G/\Gamma\to G/\Gamma_i$. Since each $G/\Gamma_i$ is a Lie group (and thus second-countable), each $\Gamma_i$ is countably-definable by Fact \ref{fact:QLT}$(c)$.  

Now, given $i\in I$, let $H_i\leq G$ be the pullback of $(G/\Gamma_i)^0$  under $G\to G/\Gamma_i$ (so $(G/\Gamma_i)^0=H_i/\Gamma_i$). Then for any $i\in I$, $H_i/\Gamma_i$ is a clopen finite-index normal subgroup of $G/\Gamma_i$.\footnote{If $K$ is a Lie group then it is locally connected, and so $K^0$ is clopen. Thus, if $K$ is also compact, then $K^0$ has finite index.} So each $H_i$ is a finite-index normal subgroup of $G$ and, since $G\to G/\Gamma_i$ is definable, $H_i$ is definable by Remark \ref{rem:defmap}$(a)$. If $i,j\in I$ and $i\leq j$, then the pullback of $H_i/\Gamma_i$ to $G/\Gamma_j$ is a clopen finite-index normal subgroup, and thus contains $H_j/\Gamma_j$, which implies $H_j\leq H_i$. Finally, $\tau_i((G/\Gamma)^0)=H_i/\Gamma_i$ for all $i\in I$ by Fact \ref{fact:compactG}$(a)$, and so $(G/\Gamma)^0\cong\varprojlim H_i/\Gamma_i$ by \cite[Corollary 1.1.8]{RZbook}. This finishes the proof of claims $(i)$ through $(iii)$

Now we deal with the remaining claims. Claim $(a)$ follows from Facts \ref{fact:compactG}$(h)$ and \ref{fact:QLT}$(c)$, and claim $(e)$ follows from Fact \ref{fact:compactG}$(g)$. For claim $(b)$, note that the quotient map $G\to G/\Gamma$ is $\theta$-definable by  Fact \ref{fact:QLT}$(d)$, and thus the remaining claims are evident from the above arguments. For claim $(c)$, if $G/\Gamma$ is profinite then each group $H_i/\Gamma_i$ is trivial by  Fact \ref{fact:compactG}$(b,c)$. Finally, for claim $(d)$, assume $G=M^*$ and suppose $G$ is pseudofinite. Then $(G/\Gamma)^0$ is abelian by Theorem \ref{thm:comm}. So, for any $i\in I$, $H_i/\Gamma_i$ is a compact connected abelian Lie group, and thus $H_i/\Gamma_i\cong \T^{n_i}$ for some $n_i\in\N$ by Fact \ref{fact:compactG}$(g)$. 
\end{proof}

\subsection{NIP formulas in pseudofinite groups} \label{sec:prePF}

In this subsection, we assume that $\cL$ expands the group language. Let $G$ be a sufficiently saturated expansion of a group.

\begin{definition}\label{def:NIP}
Let $\theta(x;\ybar)$ be a formula. Given $k\geq 1$, $\theta(x;\ybar)$ is \textbf{$k$-NIP} if there do not exist sequences $(a_i)_{i\in [k]}$ in $G$ and $(\bbar_I)_{I\seq[k]}$ in $G^{\ybar}$ such that $\theta(a_i,\bbar_I)$ holds if and only if $i\in I$. We say $\theta(x;\ybar)$ is \textbf{NIP} if it is $k$-NIP for some $k\geq 1$.
\end{definition}

\begin{remark}\label{rem:NIPVCform}
Given $k\geq 1$, a formula $\theta(x;\ybar)$ is $k$-NIP if and only if the set system $\{\theta(G;\bbar):\bbar\in G^{|\ybar|}\}$ on $G$ has VC-dimension at most $k-1$ (see, e.g., \cite[Section 6.1]{Sibook} for details on set systems and VC-dimension).
\end{remark}

Next, we summarize several main results from \cite{CPpfNIP}, which will form the basis for our work on NIP sets in finite and pseudofinite groups.  We say that a set $A\seq G$ is \textbf{generic} if $G$ is covered by finitely many left translates of $A$; and a formula $\phi(x)$ is \textbf{generic} if $\phi(G)$ is generic. 
Given a formula $\theta(x;\ybar)$,  let $S_{\theta}(G)$ denote the space of complete $\theta$-types over $G$ (i.e., ultrafilters over the Boolean algebra of $\theta$-formulas). A type $p\in S_{\theta}(G)$ is \textbf{generic} if every formula in $p$ is generic.

\begin{definition}\label{def:G00dr}
Given a formula $\theta(x;\ybar)$, we let $\theta^r(x;\ybar,u)$ denote the formula $\theta(x\cdot u;\ybar)$, and $G^{00}_{\theta^r}$ denote the intersection of all $\theta^r$-type-definable bounded-index subgroups of $G$.
\end{definition}

 Note that if $\theta(x;\ybar)$ is invariant, then so is $\theta^r(x;\ybar,u)$. One can also show that  if $\theta(x;\ybar)$ is invariant and NIP, and $\phi(x)$ is a $\theta^r$-formula, then $\phi(y \cdot x)$ is NIP.\footnote{However, $\theta^r(x;\ybar,u)$ itself need not be NIP (see \cite[Example 3.7]{CPpfNIP}).} We now focus on the case when $G$ is pseudofinite which, by saturation, implies that $G$ is an elementary extension of an ultraproduct $\prod_{\cU}G_i$, where each $G_i$ is a finite $\cL$-structure expanding a group and $\cU$ is a (nonprincipal) ultrafilter on some index set $I$.  In this case, $\mu$ will always denote the \textbf{pseudofinite counting measure} on $G$, which is obtained by lifting (e.g., as in \cite[Section 2]{HPP}) the ultralimit measure $\lim_{\cU}\mu_i$ on $\prod_{\cU}G_i$, where $\mu_i$ is the normalized counting measure on $G_i$.

 \begin{remark}\label{rem:Losmu}
  In several proofs, we will apply {\L}o\'{s}'s Theorem to properties of $\mu$, which requires an expanded language $\cL^+$ containing a sort for the ordered interval $[0,1]$ with a distance function, and functions from the $G$-sort to $[0,1]$ giving the measures of $\cL$-formulas. There are many accounts of this kind of formalism, and so we will omit further details and refer the reader to similar treatments in the literature, for example  \cite[Section 2.3]{CPpfNIP}, \cite[Section 2.6]{HruAG}, and \cite[Section 2]{HPP}.
  \end{remark}
  
  Before stating the next theorem, we ``localize" the notation given before Fact \ref{fact:QLT}. Given  a formula $\theta(x;\ybar)$, a $\theta$-type-definable set $X\seq G$, and a complete $\theta$-type $p\in S_\theta(G)$, we write $p\models X$ if $p$ contains a bounded set $\{\phi_i(x)\}_{i\in I}$ of $\theta$-formulas such that $X=\bigcap_{i\in I}\phi_i(G)$.

\begin{theorem}\label{thm:G00CP}\textnormal{\cite{CPpfNIP}}
Let $G$ be a sufficiently saturated pseudofinite expansion of a group, and suppose $\theta(x;\ybar)$ is an invariant NIP formula.
\begin{enumerate}[$(a)$]
\item A $\theta^r$-formula $\phi(x)$ is generic if and only if $\mu(\phi(x))>0$.
\item There are generic $\theta^r$-types in $S_{\theta^r}(G)$.
\item $G^{00}_{\theta^r}$ is $\theta^r$-countably-definable and normal of bounded index. 
\item Suppose $A\seq G$ is $\theta^r$-definable, and let $E\seq G/G^{00}_{\theta^r}$ be the set of $C\in G/G^{00}_{\theta^r}$ such that $p\models C\cap A$ and $q\models C\cap (G\backslash A)$
for some generic $p,q\in S_{\theta^r}(G)$. Then $E$ is closed  and $\eta_{G/G^{00}_{\theta^r}}(E)=0$. 
\end{enumerate}
\end{theorem}
\begin{proof}
See \cite[Corollary 2.15, Proposition 3.12$(a)$]{CPpfNIP} for part $(a)$.
Parts $(b)$, $(c)$, and $(d)$ are  Proposition 3.12$(c)$, Theorem 3.15$(a,e)$, and Theorem 5.2 in \cite{CPpfNIP}.
%
%
\end{proof}

\begin{remark}\label{rem:wide}
In the setting of the previous fact, part $(a)$ implies that if $X\seq G$ is $\theta^r$-type-definable, then $p\models X$ for some generic $p\in S_{\theta^r}(G)$ if and only if $\mu(W)>0$ for any definable (equivalently, any $\theta^r$-definable) set $W\seq G$ containing $X$. 
\end{remark}

We now prove the main result of this subsection. 

\begin{lemma}\label{lem:oneY}
Let $G$ be a sufficiently saturated pseudofinite expansion of a group. Suppose $\theta(x;\ybar)$ is an invariant NIP formula, and let $(W_i)_{i=0}^\infty$ be a decreasing sequence of definable sets such that $G^{00}_{\theta^r}=\bigcap_{i=0}^\infty W_i$. Fix a $\theta^r$-definable set $A\seq G$. Then, for any $\epsilon>0$, there is a $\theta^r$-definable set $Z\seq G$ and some $i\in\N$ such that $\mu(Z)<\epsilon$ and, for any $g\in G\backslash Z$, either $\mu(gW_i\cap A)=0$ or $\mu(gW_i\backslash A)=0$.
\end{lemma}
\begin{proof}
To ease notation, let $\Gamma=G^{00}_{\theta^r}$ and $K=G/\Gamma$, and let $\pi\colon G\to K$ be the canonical homomorphism. Let $E$ be as in Theorem \ref{thm:G00CP}$(d)$ (with respect to the fixed $\theta^r$-definable set $A$). So $E$ is closed and $\eta_K(E)=0$. 

Fix $\epsilon>0$. By Fact \ref{fact:QLT}$(e)$, we may fix a definable set $Z'\seq G$ such that $\mu(Z')<\epsilon$ and $\pi\inv(E)\seq Z'$. By saturation and Fact \ref{fact:QLT}$(d)$, there a s $\theta^r$-definable set $Z\seq G$ such that $\pi\inv(E)\seq Z\seq Z'$. In particular, we still have $\mu(Z)<\epsilon$. 
We show that $Z$ satisfies the conclusion of the lemma. To motivate the argument, we first make a side remark. By saturation and Remark \ref{rem:wide}, we immediately have that for any $g\in G\backslash Z$, there is some $i\in\N$ such that $\mu(gW_i\cap A)=0$ or $\mu(gW_i\backslash A)=0$. Thus the content of the following argument is to show that we can pick one $i\in\N$ that works for every $g\in G\backslash Z$.

Toward a contradiction, suppose that for all $i\in\N$ there is some $a_i\in G\backslash Z$ such that $\mu(a_iW_i\cap A)>0$ and $\mu(a_iW_i\backslash A)>0$. Let $U=\{C\in K:C\seq Z\}$, which is open in $K$ by Fact \ref{fact:QLT}$(b)$. Note that  $E\seq U$ and $\pi\inv(U)\seq Z$. In particular,  $(a_i\Gamma)_{i=0}^\infty$ is an infinite sequence in $K\backslash U$. Since $U$ is open and $K$ is compact and second-countable (by Fact \ref{fact:QLT}$(c)$ and Theorem \ref{thm:G00CP}$(c)$), we may pass to a subsequence and assume that $(a_i\Gamma)_{i=0}^\infty$ converges to some $a\Gamma\in K\backslash U$. In particular, $a\Gamma\not\in E$. \medskip

\noindent\emph{Claim}: For all $i\in \N$, $\mu(aW_i\cap A)>0$ and $\mu(aW_i\backslash A)>0$.

\noindent\emph{Proof}: First, given $i\in \N$, let $U_i=\{C\in K:C\seq aW_i\}$. As above, each $U_i$ is open in $K$. Moreover, for any $i\in\N$, since $\Gamma\seq W_i$, we have $a\Gamma\in U_i\seq aW_i/\Gamma$. Also, since $(W_i)_{i=0}^\infty$ is decreasing and $\Gamma=\bigcap_{i=0}^\infty W_i$ is a subgroup of $G$, it follows from saturation that, for all $i\in\N$, there is $n_i\in\N$ such that $W_{n_i}W_{n_i}\seq W_i$. 

Now fix $i\in \N$. Since $U_{n_i}$ is an open neighborhood of $a\Gamma$,  there is $j\geq n_i$ such that $a_j\Gamma\in U_{n_i}$.  In particular, $a_j\in \pi\inv(U_{n_i})\seq aW_{n_i}$. Now we have $a_jW_j\seq a_jW_{n_i}\seq aW_{n_i}W_{n_i}\seq aW_i$. Since $\mu(a_jW_j\cap A)>0$ and $\mu(a_jW_j\backslash A)>0$, we have $\mu(aW_i\cap A)>0$ and $\mu(aW_i\backslash A)>0$.\claim\medskip

Now, since $(W_i)_{i=0}^\infty$ is decreasing and $\Gamma=\bigcap_{i=0}^\infty W_i$, it follows from saturation that any definable set containing $a\Gamma\cap A$ (respectively, $a\Gamma\backslash A$) contains $aW_i\cap A$ (respectively, $aW_i\backslash A$) for some $i\in\N$. So $a\Gamma\in E$ by the claim (and Remark \ref{rem:wide}), which is a contradiction. 
\end{proof}

Note that Lemma \ref{lem:oneY} provides a ``regularity" statement for invariant NIP formulas in pseudofinite groups. However, it does not provide any information about the shape of the definable sets $W_i$, other than that they approximate the subgroup $G^{00}_{\theta^r}$. Indeed, most of the remaining work in this paper is toward replacing these definable sets with ones that enjoy meaningful algebraic properties.

\begin{remark}
Except for Sections \ref{sec:distal} and \ref{sec:padic}, our use of  finite VC-dimension is based entirely on applications of Lemma \ref{lem:oneY}. So it is worth emphasizing that the proof of Theorem \ref{thm:G00CP} in \cite{CPpfNIP} heavily uses fundamental results about VC-dimension such as the Sauer-Shelah Lemma and  the VC-Theorem (see \cite[Chapter 6]{Sibook}), as well as work of Simon \cite{SimRC, SimGCD} on NIP formulas, which builds on a large body of research on NIP theories and groups definable in such theories (e.g., \cite{ChSi, HPP, HP}).  
\end{remark}

\subsection{NIP subsets of groups}\label{sec:NIPsetsG}

In this subsection, we briefly recall the notion of an NIP subset of an arbitrary group, and clarify the relationship to VC-dimension and NIP formulas. Recall from the introduction that if $G$ is a group and $A\seq G$, then we have the bipartite graph $\Gamma_G(A)=(G,G;E)$ where $E=\{(x,y)\in G^2:yx\in A\}$. 

\begin{definition}
Given a group $G$ and an integer $k\geq 1$, we say that $A$ is \textbf{$k$-NIP} (in $G$) if $\Gamma_G(A)$ is $k$-NIP, i.e., it omits the bipartite graph $([k],\cP([k]);\in)$ as an induced subgraph (where $[k]=\{1,\ldots,k\}$).
\end{definition}

\begin{remark}
Fix a group $G$ and a subset $A\seq G$.
\begin{enumerate}
\item If $\Gamma_G(A)$ omits \emph{some} finite bipartite graph $(V,W,E)$ as an induced subgraph, then $A$ is $k$-NIP for some $k\leq |V|+\lceil\log_2|W|\rceil$. 
\item Given $k\geq 1$, $A$ is $k$-NIP if and only if the formula $\theta(x;y):=A(y\cdot x)$ is $k$-NIP (here we view $G$ as a structure in the group language expanded by a predicate for $A$). 
\item Given $k\geq 1$, $A$ is $k$-NIP if and only if the set system $\{gA:g\in G\}$ of left translates of $A$ has VC-dimension at most $k-1$ (see Remark \ref{rem:NIPVCform}).
\end{enumerate}
\end{remark}

Note that the formula $\theta(x;y)$ defined in the second remark is invariant. This is the main reason we define $\Gamma_G(A)$ in terms of the edge relation $yx\in A$ and not $xy\in A$. However, an important fact is that if  a bipartite graph $(V,W;E)$ is $k$-NIP then the ``opposite graph" $(W,V;\{(w,v):E(v,w)\})$ is $2^k$-NIP (see \cite[Lemma 6.3]{Sibook}).\footnote{On the other hand, if a bipartite graph is $k$-stable (as defined in the introduction) then one can easily check that the opposite graph is as well. This reconciles our definitions with those in \cite{CPT}, where we defined $k$-stable subsets of groups using the relation $xy\in A$.} So the order of the group operation when defining $k$-NIP sets only affects the precise value of $k$, and not whether the set is NIP overall.

\section{Structure and regularity: the profinite case}\label{sec:prof}
In this section, we prove a structure and regularity theorem for $\theta^r$-definable sets in a sufficiently saturated pseudofinite group $G$, where $\theta(x;\ybar)$ is an invariant NIP formula and $G/G^{00}_{\theta^r}$ is profinite (Theorem \ref{thm:UPprof} below). As an application, we obtain a structure and regularity theorem for NIP sets in finite groups of uniformly bounded exponent. We also view Theorem \ref{thm:UPprof} as a warm-up to the general case. Indeed, this theorem follows almost immediately from Lemma \ref{lem:oneY} and the fact that profinite quotients correspond to type-definable subgroups that are intersections of definable subgroups (via Lemma \ref{lem:Lie}$(c)$).

\begin{theorem}\label{thm:UPprof}
Let $G$ be a sufficiently saturated pseudofinite expansion of a group, and suppose $\theta(x;\ybar)$ is an invariant NIP formula. Assume $G/G^{00}_{\theta^r}$ is profinite. Fix a $\theta^r$-definable set $A\seq G$ and some $\epsilon>0$. Then there are
\begin{enumerate}[\hspace{5pt}$\ast$]
\item a $\theta^r$-definable finite-index normal subgroup $H\leq G$, and 
\item a set $Z\seq G$, which is a union of cosets of $H$ with $\mu(Z)<\epsilon$,
\end{enumerate}
satisfying the following properties. 
\begin{enumerate}[$(i)$]
\item \textnormal{(structure)} There is a set $D\seq G$, which is a union of cosets of $H$, such that 
\[
\mu((A\backslash Z)\smd D)=0.
\]
\item \textnormal{(regularity)} For any $g\in G\backslash Z$, either $\mu(gH\cap A)=0$ or $\mu(gH\backslash A)=0$.
\end{enumerate} 
\end{theorem}
\begin{proof}
By Theorem \ref{thm:G00CP}$(c)$ and  Lemma \ref{lem:Lie}$(a,c)$, there is a decreasing sequence $(H_i)_{i=0}^\infty$ of $\theta^r$-definable finite-index normal subgroups of $G$ such that $G^{00}_{\theta^r}=\bigcap_{i=0}^\infty H_i$. By Lemma \ref{lem:oneY}, there is a $\theta^r$-definable set $Z'\seq G$ and some $i\in\N$ such that $\mu(Z')<\epsilon$ and, if $H:=H_i$, then for any $g\in G\backslash Z'$, either $\mu(gH\cap A)=0$ or $\mu(gH\backslash A)=0$. Let $Z=\{g\in G:gH\seq Z'\}$. Then $Z$ is a union of cosets of $H$, and thus is $\theta^r$-definable since  $[G:H]$ is finite. Note that $\mu(Z)<\epsilon$ since $Z\seq Z'$. Moreover, if $g\in G\backslash Z$ then $gH=g'H$ for some $g'\in G\backslash Z'$, and so we have condition $(ii)$.  Now let $D=\bigcup\{gH:g\in G\backslash Z\text{ and }\mu(gH\cap A)>0\}$. Then $D$ is $\theta^r$-definable since $[G:H]$ is finite. Moreover, $\mu((A\backslash Z)\smd D)=0$ by condition $(ii)$ and since $[G:H]$ is finite. So we have condition $(i)$.
\end{proof}

We now prove structure and regularity for NIP sets in finite groups of \emph{uniformly bounded exponent}. In this case, we obtain the optimal situation where NIP sets are entirely controlled by finite-index subgroups up to small error.  This is related to a similar result of Alon, Fox, and Zhao \cite[Theorem 1.1]{AFZ} on finite \emph{abelian} groups of bounded exponent. Our result is stronger in the sense that the abelian assumption is removed and the structural conclusions are improved, but also weaker in the sense that we do not obtain explicit bounds. This is analogous to the comparison of our stable arithmetic regularity lemma in \cite{CPT} (Theorem \ref{thm:CPT1} above)  to the work of the third author and Wolf \cite{TeWo} on stable sets in $\F_p^n$.

\begin{theorem}\label{thm:regexp}
For any $k,r\geq 1$ and $\epsilon>0$, there is $n=n(k,r,\epsilon)$ such that the following holds. Suppose $G$ is a finite group of exponent $r$ and $A\seq G$ is $k$-NIP. Then there are
\begin{enumerate}[\hspace{5pt}$\ast$]
\item  a normal subgroup $H\leq G$ of index at most $n$, and
\item   a set $Z\seq G$, which is a union of cosets of $H$ with $|Z|<\epsilon|G|$,
\end{enumerate}
satisfying the following properties.
\begin{enumerate}[$(i)$]
\item \textnormal{(structure)} There is a set $D\seq G$, which is a union of cosets of $H$, such that
\[
|(A\backslash Z)\smd D|<\epsilon|H|.
\]
\item \textnormal{(regularity)} For any $g\in G\backslash Z$, either $|gH\cap A|<\epsilon|H|$ or $|gH\backslash A|<\epsilon|H|$.
\end{enumerate}
Moreover, $H$ is in the Boolean algebra generated by $\{gAh:g,h\in G\}$.
\end{theorem}
\begin{proof}
Note that condition $(ii)$ follows immediately from condition $(i)$. So suppose condition $(i)$ is false. Then we have some fixed $k,r\geq 1$ and $\epsilon>0$ such that, for all $i\in\N$, there is a finite group $G_i$ of exponent $r$, which is a counterexample. Specifically, there is a $k$-NIP subset $A_i\seq G_i$ such that, if $H\leq G_i$ is normal with index at most $i$, and $D,Z\seq G_i$ are unions of cosets of $H$ with $|Z|<\epsilon|G_i|$, then $|(A_i\backslash Z)\smd D|>\epsilon|H|$. 

Let $\cL$ be the group language with a new predicate $A$, and consider $(G_i,A_i)$ as a finite $\cL$-structure. Let $\cU$ be a nonprincipal ultrafilter on $\Z^+$, and let $G$ be a sufficiently saturated elementary extension of $M:=\prod_{\cU}(G_i,A_i)$. Let $\theta(x;y)$ be the formula $A(y\cdot x)$. Note that $\theta(x;y)$ is invariant. Moreover, by {\L}o\'{s}'s Theorem and since $M\prec G$, $\theta(x;y)$ is $k$-NIP (in $G$) and $G$ has exponent $r$. So $G/G^{00}_{\theta^r}$ is a compact  torsion group, and thus is profinite by Fact \ref{fact:compactG}$(d)$.  By Theorem \ref{thm:UPprof}, there is a $\theta^r$-definable finite-index normal subgroup $H\leq G$ and  sets $D,Z\seq G$, which are unions of cosets of $H$, such that $\mu(Z)<\epsilon$ and $\mu((A\backslash Z)\smd D)=0$.

Let $n=[G:H]$, and fix $\langle\theta^r\rangle$-formulas $\phi(x;\ybar)$, $\psi(x;\zbar)$, and $\zeta(x;\ubar)$ such that $H$, $D$, and $Z$ are defined by instances of $\phi(x;\ybar)$, $\psi(x;\zbar)$, and $\zeta(x;\ubar)$, respectively. Given $i\in\Z^+$, let $\mu_i$ be the normalized counting measure on $G_i$. Let $I$ be the set of $i\in\Z^+$ such that, for some tuples $\abar_i$, $\bbar_i$, and $\cbar_i$ from $G_i$,
\begin{enumerate}[$(i)$]
\item $\phi(x;\abar_i)$ defines a normal subgroup $H_i$ of $G_i$ of index $n$,
\item $\psi(x;\bbar_i)$ and $\zeta(x;\cbar_i)$ define sets $D_i,Z_i\seq G_i$, respectively, which are each unions of cosets of $H_i$, and
\item $\mu_i(Z_i)<\epsilon$ and $\mu_i((A_i\backslash Z_i)\smd D_i)<\frac{\epsilon}{n}$.
\end{enumerate}
Then $I\in\cU$ by {\L}o\'{s}'s Theorem and since $M\prec G$. So there is some $i\in I$ such that $i\geq n$, which contradicts the choice of $(G_i,A_i)$. 
\end{proof}

The statement of the previous result is almost identical to our result from \cite{CPT} on \emph{stable} arithmetic regularity in arbitrary finite groups (Theorem \ref{thm:CPT1} above), except for the presence of the error set $Z$. As in  \cite[Corollary 3.5]{CPT},  we can use this result to deduce a very strong \emph{graph regularity} statement for bipartite graphs defined by NIP subsets of finite groups of uniformly bounded exponent. 

Let $\Gamma=(V,W;E)$ be a finite bipartite graph and fix nonempty sets $X\seq V$ and $Y\seq W$. The \textbf{edge density} of the pair $(X,Y)$ is $\delta_\Gamma(X,Y):=|(X\times Y)\cap E|/|X\times Y|$. We say that $(X,Y)$ is \textbf{$\epsilon$-regular} if $|\delta_\Gamma(X,Y)-\delta_\Gamma(X_0,Y_0)|\leq\epsilon$ for any $X_0\seq X$ and $Y_0\seq Y$ such that $|X_0|\geq\epsilon|X|$ and $|Y_0|\geq\epsilon|Y|$. Given vertices $v\in V$ and $w\in W$, define $\deg_\Gamma(v,Y)=|\{y\in Y:E(v,y)\}|$ and $\deg_\Gamma(X,w)=|\{x\in X:E(x,w)\}|$.
Following \cite{CPT}, we say that the pair $(X,Y)$ is \textbf{uniformly $\epsilon$-good for $\Gamma$}, where $\epsilon>0$, if $|X|=|Y|$ and either:
\begin{enumerate}[$(i)$]
\item for any $x\in X$ and $y\in Y$, $\deg_\Gamma(x,Y)=\deg_\Gamma(X,y)\leq\epsilon |X|$, or
\item for any $x\in X$ and $y\in Y$, $\deg_\Gamma(x,Y)=\deg_\Gamma(X,y)\geq (1-\epsilon)|X|$.
\end{enumerate}
One can show that if $(X,Y)$ is uniformly $\epsilon^2$-good then it is $\epsilon$-regular, and either $\delta_\Gamma(X,Y)\leq\epsilon$ or $\delta_\Gamma(X,Y)\geq 1-\epsilon$. In fact, a stronger property holds: if $X_0\seq X$ and $Y_0\seq Y$ are nonempty and \emph{either $|X_0|\geq\epsilon|X|$ or $|Y_0|\geq\epsilon|Y|$}, then $\delta_\Gamma(X_0,Y_0)\leq\epsilon$ or $\delta_\Gamma(X_0,Y_0)\geq 1-\epsilon$ (see \cite[Proposition 3.4]{CPT}).

Now suppose $G$ is a finite group and $A$ is a subset of $G$. Let $\Gamma_G(A)=(G,G;E)$ where $E=\{(x,y)\in G^2:yx\in A\}$.  Given $X\seq G$ and $g\in G$, note that $\deg_{\Gamma_G(A)}(g,X)=|A\cap Xg|$ and $\deg_{\Gamma_G(A)}(X,g)=|A\cap gX|$.   We now observe that Theorem \ref{thm:regexp} implies a graph regularity statement for NIP subsets of finite groups of uniformly bounded exponent,  in which the partition is given by cosets of a normal subgroup and almost all pairs are uniformly good (and thus regular up to a change in $\epsilon$).  Given a group $G$, a normal subgroup $H\leq G$, and $C,D\in G/H$, let $C\cdot D$ denote the product of $C$ and $D$ in the quotient group $G/H$.

\begin{corollary}\label{cor:regexp}
For any $k,r\geq 1$ and $\epsilon>0$ there is $n(k,r,\epsilon)$ such that the following holds. Suppose $G$ is a finite group of exponent $r$ and $A\seq G$ is $k$-NIP. Then there is a  normal subgroup $H$ of index $n\leq n(k,r,\epsilon)$, and set $\Sigma\seq(G/H)^2$, with $|\Sigma|\leq\epsilon n^2$, such that any $(C,D)\not\in\Sigma$ is uniformly $\epsilon$-good for $\Gamma_G(A)$. 
\end{corollary}
\begin{proof}
Fix $k,r\geq 1$ and $\epsilon>0$ and let $n(k,r,\epsilon)$ be as in Theorem \ref{thm:regexp}. Fix a finite group $G$ and a $k$-NIP set $A\seq G$. Then there is a normal subgroup $H\leq G$ of index $n\leq n(k,r,\epsilon)$, and a set $\cI\seq G/H$ with $|\cI|\leq\epsilon n$, such that for any $C\not\in \cI$, either $|C\cap A|\leq\epsilon|H|$ or $|C\cap A|\geq (1-\epsilon)|H|$. Let $\Sigma=\{(C,D)\in (G/H)^2:C\cdot D\in \cI\}$. Then $\Sigma=\bigcup_{C\in G/H}\{(C,C\inv \cdot D):D\in\cI\}$, and so $|\Sigma|\leq\epsilon n^2$. Finally, if $(C,D)\not\in\Sigma$, then $(C,D)$  is uniformly $\epsilon$-good for $\Gamma_G(A)$ (this is identical to the calculation in the proof of \cite[Corollary 3.5]{CPT}, and makes crucial use of normality of $H$).   
\end{proof}

\begin{remark}\label{rem:profstab}
In Theorem \ref{thm:regexp}, the assumption of uniformly bounded exponent was used to obtain a certain profinite quotient, and so it is worth reviewing this argument from a more general perspective. Specifically, fix $k\geq 1$ and consider the following property of a class $\cG$ of finite groups: $(\ast)_k$ For any sequences $(G_i)_{i=0}^\infty$ and $(A_i)_{i=0}^\infty$, where $G_i\in\cG$ and $A_i\seq G_i$ is $k$-NIP, and for any ultrafilter $\cU$ on $\N$, if $G$ is a sufficiently saturated elementary extension of $\prod_{\cU}(G_i,A_i)$, then $G/G^{00}_{\theta^r}$ is profinite, where $\theta(x;y):=A(y\cdot x)$.  Then, for any $\cG$ satisfying $(\ast)_k$ and any $\epsilon>0$, there is some $n=n(k,\epsilon,\cG)$ such that any group $G\in\cG$ and $k$-NIP set $A\seq G$ satisfy the conclusions of Theorem \ref{thm:regexp} using $n$. Indeed, Theorem \ref{thm:regexp} only uses that for any $k,r\geq 1$, the class $\cG_r$ of finite groups of exponent $r$ satisfies $(\ast)_k$, for the rather heavy-handed reason that compact  torsion groups are profinite.  

Profinite quotients also arise in the stable setting (recall that $A\seq G$ is \textbf{$k$-stable} if $\Gamma_G(A)$ omits $([k],[k];\leq)$ as an induced subgraph).  In fact, if $G$ is pseudofinite and saturated, and $\theta(x;\ybar)$ is a stable invariant formula, then the group $G/G^{00}_{\theta^r}$ is  \emph{finite} (see \cite[Corollary 3.17]{CPpfNIP}). Therefore, in this case,  the set $E$ in Theorem \ref{thm:G00CP}$(d)$ 
 is empty since it is a Haar null set in a finite group. So, if one replaces ``NIP" with ``stable" in Lemma \ref{lem:oneY} and Theorem \ref{thm:UPprof}, then one can choose the error set $Z=\emptyset$, and  a similar ultraproduct argument as in Theorem \ref{thm:regexp} yields Theorem \ref{thm:CPT1} (structure and regularity for stable subsets of finite groups).\footnote{It is worth noting that this explanation of Theorem \ref{thm:CPT1} is not a faster proof than what is done in \cite{CPT}. In particular,  \cite[Corollary 3.17]{CPpfNIP} relies on the same results from \cite{HrPiGLF} used in \cite{CPT} to directly prove Theorem \ref{thm:CPT1}. Also, pseudofiniteness is not needed to prove that $G/G^{00}_{\theta^r}$ is finite.} 
\end{remark}

\section{Bohr neighborhoods}\label{sec:Bohr}

In this section, we recall some basic definitions and facts concerning Bohr neighborhoods, and define an approximate version of Bohr neighborhoods, which we will need for later arguments involving ultraproducts.

Given a group $G$, $1_G$ denotes the identity (if $G$ is abelian we use $0_G$). 
We say that the pair $(L,d)$ is a \textbf{compact metric group} if $L$ is a compact metrizable  group and $d$ is a \emph{bi-invariant} metric on $L$ compatible with the topology. By Fact \ref{fact:compactG}$(e)$, if $L$ is \emph{any} compact second-countable group then there is a (not necessarily unique) bi-invariant metric $d$ on $L$ such that $(L,d)$ is a compact metric group.

\begin{definition}
Let $H$ be a group and let $(L,d)$ be a compact metric group. Given some  $\delta>0$ and a homomorphism $\tau\colon H\to L$, define
\[
B^L_{\tau,\delta}=\{x\in H:d(\tau(x),1_L)<\delta\}.
\] 
A set $B\seq H$ is a \textbf{$(\delta,L)$-Bohr neighborhood in $H$} if $B=B^L_{\tau,\delta}$ for some homomorphism $\tau\colon H\to L$.
\end{definition}

Our ultimate goal is to transfer Bohr neighborhoods in pseudofinite groups  to Bohr neighborhoods in finite groups. To do this, we will need to approximate Bohr neighborhoods  by definable objects. This necessitates an approximate notion of a Bohr neighborhood, which involves approximate homomorphisms of groups.

\begin{definition}\label{def:approxBohr}
Let $H$ be a group and let $(L,d)$ be a compact metric group.
\begin{enumerate}
\item Given $\delta>0$, a function $f\colon H\to L$ is a \textbf{$\delta$-homomorphism} if $f(1_H)=1_L$ and, for all $x,y\in H$, $d(f(xy),f(x)f(y))<\delta$. 
\item Given $\epsilon,\delta>0$, a set $Y\seq H$ is a \textbf{$\delta$-approximate $(\epsilon,L)$-Bohr neighborhood in $H$} if $Y=\{x\in H:d(f(x),1_L)<\epsilon\}$ for some $\delta$-homomorphism $f\colon H\to L$.
\end{enumerate}
\end{definition}

Approximate homomorphisms have been studied extensively in the literature, with a special focus on the question of when an approximate homomorphism is ``close" to an actual homomorphism. For our purposes, this is what is needed to replace approximate Bohr neighborhoods with actual Bohr neighborhoods. More precisely, we will start with a definable approximate Bohr neighborhood in a pseudofinite group, and transfer this to find an approximate Bohr neighborhood in a finite group. At this point, we will be working with an approximate homomorphism from a finite group to a compact Lie group, which is a setting  where one can always find a genuine Bohr neighborhood inside an approximate Bohr neighborhood, with a negligible loss in size.

\begin{theorem}[Alekseev, Glebski\u{\i}, \& Gordon \textnormal{\cite[Theorem 5.13]{AlGlGo}}]\label{thm:AYG}
Let $(L,d)$ be a compact metric Lie group. Then there is an $\alpha_L>0$ such that, for any $0<\delta<\alpha_L$, if $H$ is a compact group and $f\colon H\to L$ is a $\delta$-homomorphism, then there is a homomorphism $\tau\colon H\to L$ such that $d(f(x),\tau(x))< 2\delta$ for all $x\in H$.
\end{theorem}

An easy consequence is that in the setting of compact Lie groups, Bohr neighborhoods can be found inside approximate Bohr neighborhoods. Given a compact metric group $(L,d)$ and $n\geq 1$, note that we have a compact metric group $(L^n,d^n)$,  where $L^n$ is endowed the product topology and $d^n(\xbar,\ybar)=\max_{1\leq i\leq n}d(x_i,y_i)$.

\begin{corollary}\label{cor:findBohr}
Let $(L,d)$ be a compact metric Lie group. Then there is an $\alpha_L>0$ such that, if $H$ is a compact group, $n\in\N$, and $0<\delta<\alpha_L$, then any $\delta$-approximate $(3\delta,L^n)$-Bohr neighborhood in $H$ contains a $(\delta,L^n)$-Bohr neighborhood in $H$.
\end{corollary}
\begin{proof}
Fix $\alpha_L>0$ from Theorem \ref{thm:AYG}. Suppose $H$ is a compact group, and $Y\seq H$ is a $\delta$-approximate $(3\delta,L^n)$-Bohr neighborhood in $H$, for some $n\in\N$ and $0<\delta<\alpha_L$, witnessed by a $\delta$-homomorphism $f\colon H\to L^n$. We may assume $n\geq 1$. For $1\leq i\leq n$, let $f_i\colon H\to L$ be given by $f_i(x)=f(x)_i$. Then each $f_i$ is a $\delta$-homomorphism. Given $1\leq i\leq n$,  Theorem \ref{thm:AYG} provides a homomorphism $\tau_i\colon H\to L$ such that $d(f_i(x),\tau_i(x))<2\delta$ for all $x\in H$. Let $\tau\colon H\to L^n$ be such that $\tau(x)=(\tau_1(x),\ldots,\tau_n(x))$. Then $\tau$ is a homomorphism and $d^n(f(x),\tau(x))<2\delta$ for all $x\in H$. Now we have $B^{L^n}_{\tau,\delta}\seq Y$ by the triangle inequality.
\end{proof}

The next result provides a lower bound on the size of Bohr neighborhoods in finite groups. The proof is a standard averaging argument (adapted from the abelian case; see \cite[Lemma 4.20]{TaoVu} and/or \cite[Lemma 4.1]{GreenSLAG}). We include the details for the sake of clarity and to observe that the method works for Bohr neighborhoods in nonabelian finite groups defined using compact metric groups.

\begin{proposition}\label{prop:Bohrgeneric}
Let $(L,d)$ be a compact metric group and, given $\delta>0$, let $\ell_\delta=\eta_L(\{t\in L:d(t,1_L)<\delta\})$. For any finite group $H$ and $\delta>0$, if $B\seq H$ is a $(2\delta,L)$-Bohr neighborhood in $H$, then $|B|\geq \ell_\delta |H|$. 
\end{proposition}
\begin{proof}
 Fix a finite group $H$, a homomorphism $\tau\colon H\to L$, and some $\delta>0$. Given $x\in H$, let $f_x\colon L\to\{0,1\}$ be the characteristic function of $\{t\in L:d(\tau(x),t)<\delta\}$. Then
\[
\ell_\delta|H|=\sum_{x\in H}\int_{L}f_x~d\eta_L=\int_{L}\sum_{x\in H}f_x~d\eta_L.
\]
So there must be some $t\in L$ such that $\sum_{x\in H}f_x(t)\geq \ell_\delta|H|$. In other words, if $S=\{x\in H:d(\tau(x),t)<\delta\}$ then $|S|\geq \ell_\delta|H|$. Fix $a\in S$. For any $x\in S$, we have
\[
d(\tau(xa\inv),1_L)=d(\tau(x),\tau(a))\leq d(\tau(x),t)+d(\tau(a),t)<2\delta.
\]
Therefore $Sa\inv\seq B^L_{\tau,2\delta}$, and so $|B^L_{\tau,2\delta}|\geq|Sa\inv|=|S|\geq\ell_\delta |H|$.
\end{proof}

\begin{remark}\label{rem:Liemetric}
As we have seen in previous results (e.g., Lemma \ref{lem:Lie}), real tori of the form $\T^r$ have a distinguished role in the study of compact abelian groups. Moreover, in the setting of finite abelian groups, Bohr neighborhoods are usually defined using homomorphisms to the torus (see, e.g., \cite{BourgTAP}, \cite{GreenSLAG}).  Thus, in order to match these definitions  more explicitly, we define the metric $d_{\T^1}(x,y)= \min\{|x-y|,1-|x-y|\}$ on $\T^1$ (identified with $[0,1)$) and the product metric $d_{\T^r}:=d^r_{\T^1}$ on $\T^r$ for $r\in\N$. Throughout the rest of the paper, when we speak of $\T^r$ as a compact metric group, will always work with this choice of metric.
\end{remark}

\begin{definition}\label{def:BohrT}
Given a finite group $H$, a homomorphism $\tau\colon H\to\T^r$, and some $\delta>0$, we let $B^r_{\tau,\delta}$ denote $B^{\T^r}_{\tau,\delta}$. We call $B\seq H$ a \textbf{$(\delta,r)$-Bohr neighborhood in $H$} if $B=B^r_{\tau,\delta}$ for some $\tau$.\footnote{In this case, $r$ and $\delta$ are sometimes referred to as the \emph{rank} and \emph{width} of $B$, respectively.}
\end{definition}

In Theorem \ref{thm:regexp}, we showed that NIP sets in finite groups of uniformly bounded exponent are approximated by normal subgroups of uniformly bounded index. As noted in the introduction, we cannot expect such a result for NIP sets in finite groups $G$ of unrestricted exponent. So instead of subgroups of $G$, we will consider pairs $(B,H)$, where $H$ is a normal subgroup of $G$ and $B$ is a $(\delta,r)$-Bohr neighborhood in $H$. Thus, in Proposition \ref{prop:Bohrlike} below, we point out some ways in which Bohr neighborhoods behave like normal subgroups of ``small" index. We will need the following minor generalization of a well-known exercise, namely, if $G$ is an amenable group and $A\seq G$ has positive upper density then $AA\inv$ is generic. 

\begin{proposition}\label{prop:separating}
Suppose $G$ is a group, $\cB$ is a left-invariant Boolean algebra of subsets of $G$, and $\nu$ is a left-invariant  finitely additive probability measure on $\cB$. Suppose $A\in\cB$ is such that $\nu(A)>0$. Then, for any $X\seq G$, there is a finite set $F\seq X$ such that $|F|\leq\frac{1}{\nu(A)}$ and $X\seq F AA\inv$.
\end{proposition}
\begin{proof}
We say that $Y\seq G$ \emph{separates $A$} if $xA\cap yA=\emptyset$ for all distinct $x,y\in Y$. By the assumptions on $\nu$, if $Y\seq G$ separates $A$ then $|Y|\leq\frac{1}{\nu(A)}$. Choose a finite set $F\seq X$ with maximal size among subsets of $X$ that separate $A$. Fix $x\in X$. Then there is $y\in F$ such that $xA\cap yA\neq\emptyset$, and so we may fix $z\in xA\cap yA$. Then $y\inv z\in A$ and $z\inv x\in A\inv$, which means $y\inv x\in AA\inv$, and so $x\in FAA\inv$.
\end{proof}

\begin{proposition}\label{prop:Bohrlike}
Let $G$ be a finite group and $H\leq G$ be a normal subgroup of index $n$. Suppose $B$ is a $(\delta,r)$-Bohr neighborhood in $H$, where $r\in\N$ and $0<\delta\leq 2$.
\begin{enumerate}[$(a)$]
\item $B=B\inv$, $1_G\in B$, and $gB=Bg$ for any $g\in G$.
\item For any $X\seq G$ there is $F\seq X$ such that $|F|\leq n(\frac{2}{\delta})^r$ and $X\seq FB$.  Thus $G$ is covered by at most $n(\frac{2}{\delta})^r$ translates of $B$.
\end{enumerate}
\end{proposition}
\begin{proof}
Part $(a)$. We have $B=B\inv$ by bi-invariance of $d_{\T^r}$, and clearly $1_G\in B$. If $g\in G$ and $x\in B$ then $gxg\inv\in H$ (since $H$ is normal), and so $gxg\inv\in B$ by bi-invariance of $d_{\T^r}$. It follows that $gB=Bg$ for any $g\in G$.

Part $(b)$. Suppose $B=B^r_{\tau,\delta}$ for some $\tau\colon H\to \T^r$, and let $B_0=B^r_{\tau,\delta/2}$. Note that $\eta_{\T^r}(\{t\in \T^r:d_{\T^r}(t,0_{\T^r})<\frac{\delta}{4}\})=(\frac{\delta}{2})^r$, and so
$|B_0|\geq (\frac{\delta}{2})^r|H|= n\inv(\frac{\delta}{2})^r|G|$ by Proposition \ref{prop:Bohrgeneric}. Now fix $X\seq G$. By Proposition \ref{prop:separating} (with $\nu$ the normalized counting measure on $G$), we have $X\seq FB_0B_0\inv$ for some $F\seq X$ with $|F|\leq n(\frac{2}{\delta})^r$. Finally, note that if $x,y\in B_0$ then
\[
d(\tau(xy\inv),0_{\T^r})=d(\tau(x),\tau(y))\leq d(\tau(x),0_{\T^r})+d(\tau(y),0_{\T^r})<\delta.
\]
So $B_0B_0\inv\seq B$, and we have $X\seq FB$. 
\end{proof}

Note that part $(a)$ of the previous fact holds for any compact metric group $(L,d)$ in place of $\T^r$ (and does not use that $G$ is finite); and the analogue of part $(b)$ holds for any $(L,d)$ with $(\frac{\delta}{2})^r$ replaced by $\ell_{\delta/4}$ (from Proposition \ref{prop:Bohrgeneric}). We refer the reader to \cite[Section 4.4]{TaoVu} and \cite[Sections 3 \& 4]{GreenSLAG} for more on the role of Bohr neighborhoods in arithmetic combinatorics and discrete Fourier analysis.

\section{Structure and regularity: the general case}\label{sec:gen}

The next goal is a result analogous to Theorem \ref{thm:UPprof}, but without the assumption that $G/G^{00}_{\theta^r}$ is profinite. For this, we need  to understand more about descriptions of $G^{00}_{\theta^r}$ as an intersection of definable subsets of $G$. The goal is to find properties of definable sets which are both interesting algebraically, and also sufficiently first-order so that they can be transferred to finite groups in arguments with ultraproducts. In particular, we will use approximate Bohr neighborhoods. 

Suppose $G$ is a group definable in a sufficiently saturated structure, and $\Gamma\leq G$ is a type-definable normal subgroup of bounded index. By Lemma \ref{lem:Lie}, $\Gamma$ is an intersection of a bounded number of definable finite-index normal subgroups of $G$ whenever $G/\Gamma$ is profinite. The next result shows that in general, we can write $\Gamma=\bigcap_{i\in I}W_i$ where each $W_i$ is a definable subset of a definable finite-index normal subgroup $H_i\leq G$, and there is a Bohr neighborhood $B_i$ in $H_i$ such that $\Gamma\seq B_i\seq W_i$. Moreover, $B_i$ is obtained from a definable homomorphism to a compact connected Lie group (so, in particular, $B_i$ is \emph{co-type-definable}).

\begin{proposition}\label{prop:Bohr}
Let $G$ be a group definable in a sufficiently saturated structure $M^*$.
Suppose $\Gamma\leq G$ is type-definable and normal of bounded index. Then there is a bounded jsl $I$ and a decreasing net $(W_i)_{i\in I}$ of definable subsets of $G$ such that $\Gamma=\bigcap_{i\in I} W_i$ and, for all $i\in I$, there are
\begin{enumerate}[\hspace{5pt}$\ast$]
\item a definable finite-index normal subgroup $H_i\leq G$,
\item a definable homomorphism $\pi_i\colon H_i\to L_i$, where $(L_i,d_i)$ is a compact connected metric Lie group, and
\item a real number $\delta_i>0$,  
\end{enumerate}
such that $\Gamma\seq \ker\pi_i\seq B^{L_i}_{\pi_i,\delta_i}\seq W_i\seq H_i$. Moreover:
\begin{enumerate}[$(a)$]
\item If $\Gamma$ is countably-definable then we may assume $I=\N$.
\item If $G=M^*$  and $\Gamma$ is $\theta$-type-definable for some invariant formula $\theta(x;\ybar)$, then we may assume $W_i$, $H_i$, and $\pi_i$ are $\theta$-definable.
\item If $G=M^*$ and $G$ is pseudofinite, then we may assume $L_i=\T^{n_i}$ for some $n_i\in\N$.
\item If $G/\Gamma$ is abelian then we may assume $H_i=G$ and $L_i=\T^{n_i}$ for some $n_i\in\N$.
\end{enumerate}
\end{proposition}
\begin{proof}
Let $(\Gamma_i)_{i\in I_0}$ and $(H_i)_{i\in I_0}$ be as in Lemma \ref{lem:Lie}, where $I_0$ is a small \emph{jsl}. For each $i\in I_0$, let $L_i=H_i/\Gamma_i$ and equip $L_i$ with some bi-invariant metric $d_i$ (by Fact \ref{fact:compactG}$(e)$).  Let $\pi_i:H_i\to L_i$ be the canonical homomorphism. Then $\pi_i$ is definable since $G\to G/\Gamma_i$ is definable. 

For each $i\in I_0$, let $(W^i_n)_{n=0}^\infty$ be a decreasing sequence of definable subsets of $G$ such that $\Gamma_i=\bigcap_{n=0}^\infty W^i_n$. Let $I$ be the set of all finite subsets of $I_0\times\N$, and view $I$ as a \emph{jsl} under the subset ordering. Given $\sigma=\{(i_1,n_1),\ldots,(i_k,n_k)\}\in I$, let $i_\sigma=\sup\{i_1,\ldots,i_k\}$, and set $H_\sigma:=H_{i_\sigma}$, $W_\sigma:=H_\sigma\cap\bigcap_{t=1}^k W_{n_t}^{i_t}$,  $\Gamma_\sigma:=\Gamma_{i_\sigma}$, $L_\sigma=L_{i_\sigma}$, and $\pi_\sigma=\pi_{i_\sigma}$. Note that $\Gamma_\sigma\seq W_\sigma\seq H_\sigma$ for all $\sigma\in I$, and $\bigcap_{\sigma\in I}W_\sigma=\Gamma$. By choice of $i_\sigma$, we also have that $(W_\sigma)_{\sigma\in I}$ is decreasing.

Now, given  $\sigma\in I$, the set $U_\sigma=\{a\Gamma_{\sigma}\in G/\Gamma_{\sigma}:a\Gamma_{\sigma}\seq W_\sigma\}$ is an identity neighborhood in $L_{\sigma}$ (by  Fact \ref{fact:QLT}$(b)$) and, by construction, $\pi_{\sigma}\inv(U_\sigma)\seq W_\sigma$. So choose $\delta_\sigma>0$ such that $U_\sigma$ contains the open ball of radius $\delta_\sigma$ around $1_{L_\sigma}$. Note that $\Gamma\seq\Gamma_\sigma=\ker\pi_\sigma\seq B^{L_\sigma}_{\pi_\sigma,\delta_\sigma}\seq W_\sigma$. 
This finishes the proof of the main statement. 

We now deal with the remaining claims. Claims $(b)$ and $(c)$ follow by applying Lemma \ref{lem:Lie}$(b,d)$ in the above construction. 
For claim $(a)$, suppose $\Gamma$ is countably-definable. Then, by Lemma \ref{lem:Lie}$(a)$, we may assume $I_0=\N$. So $I$ is the \emph{jsl} of finite subsets of $\N\times\N$ under the subset ordering, which contains a cofinal sub-\emph{jsl} isomorphic to $\N$. 

Finally, for claim $(d)$, suppose $G/\Gamma$ is abelian. Then $G/\Gamma_i$ is a compact abelian Lie group for any $i\in I_0$, and thus isomorphically embeds in $\T^{n_i}$ for some $n_i\in\N$ by Fact \ref{fact:compactG}$(g)$. So in the argument above, we can replace each $H_i$ and $\pi_i$ with $G$ and $G\to G/\Gamma_i\seq \T^{n_i}$, respectively. 
\end{proof}

Note that in the previous result, the homomorphism $\pi_i\colon H_i\to L_i$ is also surjective, except in part $(d)$ where we can replace $H_i$ with $G$ when $G/\Gamma$ is abelian. The purpose of part $(d)$ is to note that if $G$ is already abelian then one can obtain Bohr neighborhoods defined by tori without first passing to a subgroup $H_i$, and without  Theorem \ref{thm:comm} or an extra pseudofiniteness assumption.

One drawback of Proposition \ref{prop:Bohr} is that the Bohr neighborhood $B^{L_i}_{\pi_i,\delta_i}$ is not necessarily definable. In order to work with definable objects, we will have to consider approximate Bohr neighborhoods.

\begin{definition}\label{def:Bohrchain}
Let $G$ be a group definable in a sufficiently saturated structure $M^*$. Suppose $H\leq G$ is definable and $\pi\colon H\to L$ is a definable homomorphism to a compact metric group $(L,d)$. Given $t\geq 1$, we say that a sequence $(Y_m)_{m=0}^\infty$ of subsets of $H$ is a \textbf{definable $(t,\pi)$-approximate Bohr chain in $H$} if $(Y_m)_{m=0}^\infty$ is decreasing, $\ker\pi=\bigcap_{i=0}^\infty Y_m$, and there are  $(\delta_m)_{m=0}^\infty$ and $(f_m)_{m=0}^\infty$ such that:
\begin{enumerate}[$(i)$]
\item $(\delta_m)_{m=0}^\infty$ is a decreasing sequence positive real numbers converging to $0$,
\item for all $m$, $f_m\colon H\to L$ is a definable $\delta_m$-homomorphism with finite image, and
\item for all $m$, $Y_m=\{x\in H:d(f_m(x),1_L)<t\delta_m\}$.
\end{enumerate}
Moreover,  if $G=M^*$ and $H$, $\pi$, and $f_m$ are all $\theta$-definable, for some fixed formula $\theta(x;\ybar)$, then we say ``$\theta$-definable" in place of ``definable".
\end{definition}

\begin{proposition}\label{prop:Ydef} 
Suppose $(Y_m)_{m=0}^\infty$ is a definable (respectively, $\theta$-definable) $(t,\pi)$-approximate Bohr chain in $H\leq G$, as in Definition \ref{def:Bohrchain}. Then each set $Y_m$ is  a $\delta_m$-approximate $(t\delta_m,L)$-Bohr neighborhood in $H$, and a definable (respectively, $\theta$-definable) subset of $G$.\footnote{On the other hand, we are \emph{not} claiming that the family $\{Y_m:m\in\N\}$ is \emph{uniformly} definable (which could be misconstrued from our choice of terminology).}
\end{proposition}
\begin{proof}
The first claim is obvious. For definability, note that for any $m\in\N$, $H$ is partitioned into finitely many fibers $f_m\inv(\lambda)$ for $\lambda\in f_m(H)$, each of which is definable (respectively, $\theta$-definable) by Remark \ref{rem:defmap}. Now $Y_m$ is a union of the (finitely many) fibers $f_m\inv(\lambda)$ where $\lambda\in f(H)$ is such that $d(\lambda,1_L)<\epsilon$.
\end{proof}

The parameter $t$ in Definition \ref{def:Bohrchain} is introduced in order to control the ``width" $\epsilon$ and the ``error" $\delta$ in a $\delta$-approximate $(\epsilon,L)$-Bohr neighborhood. Specifically, it is desirable to have $\epsilon$ be some constant multiple $t$ of $\delta$, and in the following results we will choose $t$ arbitrarily. This will eventually be used to find actual Bohr neighborhoods inside approximate Bohr neighborhoods, with $L=\T^n$ for some $n\in\N$, in which case setting $t=3$ will suffice (via Corollary \ref{cor:findBohr}).

\begin{lemma}\label{lem:Bohrapprox}
Let $G$ be a group definable in a sufficiently saturated structure $M^*$. Suppose $H$ is a definable subgroup of $G$ and $\pi\colon H\to L$ is a definable homomorphism to a compact metric group $(L,d)$. Then, for any $t\geq 1$, there is a definable $(t,\pi)$-approximate Bohr chain $(Y_m)_{m=0}^\infty$ in $H$. Moreover, if $G=M^*$ and $H$ and $\pi$ are $\theta$-definable for some invariant formula $\theta(x;\ybar)$, then $(Y_m)_{m=0}^\infty$ is a $\theta$-definable $(t,\pi)$-approximate Bohr chain. 
\end{lemma}
\begin{proof}
Given $\lambda\in L$ and $\epsilon>0$, let $K(\lambda,\epsilon)\seq L$ and $U(\lambda,\epsilon)\seq L$ be the closed ball of radius $\epsilon$ around $\lambda$ and the open ball of radius $\epsilon$ around $\lambda$, respectively. 

Fix $m\geq 1$. Choose a finite set $\Lambda\seq L$ such that $L=\bigcup_{\lambda\in\Lambda}K(\lambda,\frac{1}{2m})$ and $1_L\in\Lambda$. For any $\lambda\in\Lambda$, since $\pi$ is definable, there is a definable set $D_\lambda\seq H$ such that $\pi\inv(K(\lambda,\frac{1}{2m}))\seq D_\lambda\seq\pi\inv(U(\lambda,\frac{1}{m}))$ (see Remark \ref{rem:defmap}$(a)$). Enumerate $\Lambda=\{\lambda_1,\ldots,\lambda_k\}$, with $\lambda_1=1_L$, and, for $1\leq i\leq k$, let $D_i=D_{\lambda_i}$. For $1\leq i\leq k$, define $E_i=D_i\backslash \bigcup_{j<i}D_j$. Then $E_1,\ldots,E_k$ are definable and partition $H$. This determines a definable function $f_m\colon H\to \Lambda$ such that $f_m(x)=\lambda_i$ if and only if $x\in E_i$. For any $x\in H$, we have $x\in D_{f_m(x)}\seq\pi\inv(U(f_m(x),\frac{1}{m}))$, and so $d(\pi(x),f_m(x))<\frac{1}{m}$. Note that $f_m(1_H)=1_L$ by definition. Also, given $x,y\in H$, we have
\begin{align*}
d(f_m(xy),f_m(x)f_m(y)) &\leq d(f_m(xy),\pi(xy))+d(\pi(x)\pi(y),f_m(x)\pi(y))\\
 &\hspace{1.52in}+d(f_m(x)\pi(y),f_m(x)f_m(y))\\
 &= d(f_m(xy),\pi(xy))+d(\pi(x),f_m(x))+d(\pi(y),f_m(y))<\textstyle\frac{3}{m}.
\end{align*}
Altogether, $f_m\colon H\to L$ is a $\frac{3}{m}$-homomorphism.

 Now fix an integer $t\geq 1$.  For $m\in\N$, define
\[
Y_m=\textstyle\left\{x\in H:d(f_m(x),1_L)<\frac{3t}{m}\right\}.
\] 
Note that $D_1\seq Y_m$, and so $\ker\pi\seq\pi\inv(K(1_L,\frac{1}{2m}))\seq Y_m$. We now have a sequence $(Y_m)_{m=1}^\infty$ of definable subsets of $H$, with $\ker\pi \seq Y_m$ for all $m\in\N$.  Moreover, for any $m\in\N$, if $x\in Y_m$ then 
\[
d(\pi(x),1_L)\leq d(\pi(x),f_m(x))+d(f_m(x),1_L)<\textstyle\frac{3t+1}{m}.
\]
This implies $\ker\pi=\bigcap_{m=0}^\infty Y_m$. Finally, given $m\in\N$, we have $\pi\inv(K(1_L,\frac{1}{2m}))\seq Y_m\seq\pi\inv (U(1_L,\frac{3t+1}{m}))$. In particular, if $n\geq (6t+2)m$, then $Y_n\seq Y_m$. So, after thinning the sequence, we may assume $Y_{m+1}\seq Y_m$ for all $m\in\N$. Altogether, if $\delta_m=\frac{3}{m}$, then $(\delta_m)_{m=1}^\infty$ and $(f_m)_{m=1}^\infty$ witness that $(Y_m)_{m=1}^\infty$ is a definable $(t,\pi)$-approximate Bohr chain in $H$.

For the ``moreover" statement, suppose $G=M^*$, and $H$ and $\pi$ are $\theta$-definable for some formula $\theta(x;\ybar)$. Then one can choose each $D_\lambda$ to be $\theta$-definable (by Remark \ref{rem:defmap}). So each $f_m$ is $\theta$-definable by construction, and  $(Y_m)_{m=1}^\infty$ is a $\theta$-definable $(t,\pi)$-approximate Bohr chain.  
\end{proof}

We now combine the above ingredients to prove a structure and regularity theorem for $\theta^r$-definable sets in a sufficiently saturated pseudofinite group $G$, where $\theta(x;\ybar)$ is an arbitrary invariant NIP formula. Recall that we use $\mu$ for the pseudofinite counting measure.

\begin{theorem}\label{thm:UPgen}
Let $G$ be a sufficiently saturated pseudofinite expansion of a group, and suppose $\theta(x;\ybar)$ is an invariant NIP formula. Fix a $\theta^r$-definable set $A\seq G$ and some $\epsilon>0$. Then there are
\begin{enumerate}[\hspace{5pt}$\ast$]
\item a $\theta^r$-definable finite-index normal subgroup $H\leq G$,
\item a $\theta^r$-definable homomorphism $\pi\colon H\to\T^n$, for some $n\in\N$, and
\item a $\theta^r$-definable set $Z\seq G$, with $\mu(Z)<\epsilon$,
\end{enumerate}
such that, for any integer $t\geq 1$, there is
\begin{enumerate}[\hspace{5pt}$\ast$]
\item a $\theta^r$-definable  $(t,\pi)$-approximate Bohr chain $(Y_m)_{m=0}^\infty$ in $H$
\end{enumerate}
satisfying the following properties, for any $m\in\N$.
\begin{enumerate}[$(i)$]
\item \textnormal{(structure)} There is a set $D_m\seq G$, which is a union of finitely many left translates of $Y_m$, such that
\[
\mu((A\smd D_m)\backslash Z)=0.
\]
\item \textnormal{(regularity)} For any $g\in G\backslash Z$, either $\mu(gY_m\cap A)=0$ or $\mu(gY_m\backslash A)=0$. 
\end{enumerate}
Moreover, if $G/G^{00}_{\theta^r}$ is abelian then we may assume $H=G$.
\end{theorem}
\begin{proof}
Let $(W_i)_{i=0}^\infty$ be a sequence of $\theta^r$-definable sets in $G$ satisfying the conditions of Proposition \ref{prop:Bohr} with $\Gamma=G^{00}_{\theta^r}$.   By Lemma \ref{lem:oneY}, there is a $\theta^r$-definable set $Z\seq G$ and some $i\in\N$ such that $\mu(Z)<\epsilon$ and, if $W:=W_i$, then for all $g\in G\backslash Z$, we have $\mu(gW\cap A)=0$ or $\mu(gW\backslash A)=0$. Proposition \ref{prop:Bohr} associates to $W$ a $\theta^r$-definable homomorphism $\pi\colon H\to\T^n$, where $H$ is a $\theta^r$-definable finite-index normal subgroup of $G$. If $G/G^{00}_{\theta^r}$ is abelian then we may further assume $H=G$. Fix $t\geq 1$. By Lemma \ref{lem:Bohrapprox}, there is a $\theta^r$-definable $(t,\pi)$-approximate Bohr chain $(Y_m)_{m=0}^\infty$ in $H$. Recall that each $Y_m$ is definable by Proposition \ref{prop:Ydef}. Since $\ker\pi$ is type-definable and contained in the definable set $W$, it follows from saturation  that $Y_m\seq W$ for sufficiently large $m$. So for sufficiently large $m$ we have that, for any $g\in G\backslash Z$, either $\mu(gY_m\cap A)=0$ or $\mu(gY_m\backslash A)=0$. Thus, after removing finitely many sets $Y_m$ from the sequence, we have condition $(ii)$. 

Toward proving condition $(i)$, fix $m\in\N$. Since $\ker\pi$ is a subgroup of $H$, we may use saturation (similar to as in the proof of Lemma \ref{lem:oneY}), to find some $r\geq m$ such that $Y_r Y_r\inv\seq Y_m$. Since $Y_r$ contains a type-definable bounded-index subgroup of $G$ (namely, $\ker\pi)$, it follows that $Y_r$ is generic and so $\mu(Y_r)>0$. By Proposition \ref{prop:separating}, there is a finite set $F\seq G\backslash Z$ such that $G\backslash Z\seq FY_m$. Let $I$ be the set of $g\in F$ such that $\mu(gY_m\backslash A)=0$, and note that if $g\in F\backslash I$ then $\mu(gY_m\cap A)=0$. Let $D_m=IY_m$. Since $G\backslash Z\seq FY_m$, we have
\[
A\smd D_m\seq Z\cup \bigcup_{g\in I} (gY_m\backslash A)\cup\bigcup_{g\in F\backslash I}(gY_m\cap A),
\]
 and so $\mu((A\smd D_m)\backslash Z)=0$. 
\end{proof}

The next goal is our main result for NIP sets in arbitrary finite groups (see Theorem \ref{thm:mainNIP}). Roughly speaking, we will show that if $A$ is a $k$-NIP set in a finite group $G$, then there is a normal subgroup $H\leq G$, and Bohr neighborhood $B$ in $H$, such that almost  all translates of $B$ are almost contained in $A$ or almost disjoint from $A$ (up to some error $\epsilon>0$). Moreover, $A$ is approximately a union of translates of $B$, and the index of $H$ and  complexity of $B$ are bounded in terms of $k$ and $\epsilon$.

The proof of Theorem \ref{thm:mainNIP} from Theorem \ref{thm:UPgen} is of course in analogy to the proof of Theorem \ref{thm:regexp} from Theorem \ref{thm:UPprof}. However, the argument is significantly more complicated, due to certain crucial differences between  Bohr neighborhoods and subgroups. Specifically, Bohr neighborhoods are not closed under the group operation and, moreover, distinct translates of a Bohr neighborhood need not be disjoint.\footnote{This is also the reason why Theorem \ref{thm:mainNIP} does not yield a graph regularity statement for $\Gamma_G(A)$ involving a partition into translates of Bohr neighborhoods.} Thus, when using Theorem \ref{thm:UPgen} to obtain results for finite groups, we will focus solely on the \emph{regularity} statement. In order to then deduce \emph{structure} from \emph{regularity} in finite groups, we will need to argue similarly as in the end of the proof of Theorem \ref{thm:UPgen}, while also taking into account the quantitative behavior of Bohr neighborhoods described in Propositions \ref{prop:Bohrgeneric} and \ref{prop:Bohrlike}. So we first prove a lemma, which gives a rather flexible version of the \emph{regularity} statement, and also contains the ultraproduct argument necessary to prove Theorem \ref{thm:mainNIP}.

\begin{lemma}\label{lem:downstairs}
For any $k\geq 1$ and $\epsilon>0$, and any function $\gamma\colon (\Z^+)^2\times(0,1]\to\R^+$,  there is $n=n(k,\epsilon,\gamma)$ such that the following holds. Suppose $G$ is a finite group and $A\seq G$ is $k$-NIP. Then there are
\begin{enumerate}[\hspace{5pt}$\ast$]
\item a normal subgroup $H\leq G$ of index $m\leq n$,
\item a $(\delta,r)$-Bohr neighborhood $B$ in $H$ and a $\delta$-approximate $(3\delta,r)$-Bohr neighborhood $Y$ in $H$, where $r\leq n$ and $\frac{1}{n}\leq \delta\leq 1$, and
\item a set $Z\seq G$, with $|Z|<\epsilon|G|$,
\end{enumerate}
such that $B\seq Y\seq H$ and, for any $g\in G\backslash Z$, either $|gY\cap A|< \gamma(m,r,\delta)|B| $ or $|gY\backslash A|< \gamma(m,r,\delta)|B|$. Moreover, $H$, $Y$, and $Z$ are in the Boolean algebra generated by $\{gAh:g,h\in G\}$, and if $G$ is abelian then we may assume $H=G$.
\end{lemma}
\begin{proof}
Suppose not. Then we have $k\geq 1$, $\epsilon>0$, and  $\gamma\colon (\Z^+)^2\times(0,1]\to\R^+$ witnessing this. In particular, for any $n\geq 1$, there is a finite group $G_n$ and a $k$-NIP subset $A_n\seq G_n$ such that, for any $H,B,Y,Z\seq G_n$, if
\begin{enumerate}[\hspace{10pt}$\ast$]
\item $H$ is a normal subgroup of $G_n$ of index $m\leq n$, and $H=G_n$ if $G_n$ is abelian, 
\item $H$, $Y$, and $Z$ are in the Boolean algebra generated by $\{gA_nh:g,h\in G_n\}$,
\item $B$ is a $(\delta,r)$-Bohr neighborhood in $H$ and $Y$ is a $\delta$-approximate $(3\delta,r)$-Bohr neighborhood in $H$, where $r\leq n$ and $\frac{1}{n}\leq\delta\leq 1$,
\item $|Z|<\epsilon|G_n|$, and  $B\seq Y\seq H$, 
\end{enumerate}
then there is some $g\in G_n\backslash Z$ such that $|gY\cap A_n|\geq \gamma(m,r,\delta)|B|$ and $|gY\backslash A_n|\geq\gamma(m,r,\delta)|B|$. 

Let $\cL$ be the group language together with an extra predicate $A$, and consider each $(G_n,A_n)$ as a finite $\cL$-structure. Let $\cU$ be a nonprincipal ultrafilter on $\Z^+$, and let $G$ be a sufficiently saturated elementary extension of $M:=\prod_{\cU}(G_n,A_n)$. We identify $A=A(G)$. Let $\theta(x;y)$ be the formula $A(y\cdot x)$. Note that $\theta(x;y)$ is invariant, and $k$-NIP (in $G$) by {\L}o\'{s}'s Theorem and since $M\prec G$.  Finally, let $\alpha:=\min\{\alpha_{\T^1},1\}$, where $\alpha_{\T^1}>0$ is as in Corollary \ref{cor:findBohr}. By Theorem \ref{thm:UPgen} (with $t=3$), there are $\theta^r$-definable sets $Y,Z\seq G$ (for $Y$ we obtain $\theta^r$-definability via Proposition \ref{prop:Ydef}), a $\theta^r$-definable finite-index normal subgroup $H\leq G$, and a $\theta^r$-definable $\delta$-homomorphism $f\colon H\to \T^r$, for some $r\in\N$ and $0<\delta<\alpha$, such that:
\begin{enumerate}[\hspace{10pt}$\ast$]
\item if $G$ is abelian then $H=G$,
\item $\mu(Z)<\epsilon$,
\item $f(H)$ is finite and $Y=\{x\in H:d_{\T^r}(f(x),0_{\T^r})<3\delta\}$, and 
\item for any $g\in G\backslash Z$, either $\mu(gY\cap A)=0$ or $\mu(gY\backslash A)=0$.
\end{enumerate}

Let $m=[G:H]$, and set $\epsilon^*=\gamma(m,r,\delta)m\inv\delta^r>0$.  Let $\Lambda=f(H)$ and, given $\lambda\in\Lambda$, let $F(\lambda)=f\inv(\lambda)$. Then each $F(\lambda)$ is $\theta^r$-definable by Remark \ref{rem:defmap}. 

Fix $\langle\theta^r\rangle$-formulas $\phi(x;\ybar)$, $\psi(x;\zbar)$, $\zeta(x;\ubar)$, and $\xi_\lambda(x;\vbar_\lambda)$ for $\lambda\in\Lambda$ (without parameters) such that $H$, $Y$, $Z$, and $F(\lambda)$ for $\lambda\in\Lambda$ are defined by instances of $\phi(x;\ybar)$, $\psi(x;\zbar)$, $\zeta(x;\ubar)$, and $\xi_\lambda(x;\vbar_\lambda)$,  respectively. Given $n\in\N$, let $\mu_n$ denote the normalized counting measure on $G_n$. Let $d$ denote $d_{\T^r}$. Define $I\seq\Z^+$ to be the set of $n\in \Z^+$ such that, for some tuples $\abar_n$, $\bbar_n$, $\cbar_n$, and $\dbar_{n,\lambda}$ for $\lambda\in\Lambda$, 
\begin{enumerate}[$(i)$]
\item $\phi(x;\abar_n)$ defines a normal subgroup $H_n$ of $G_n$ of index $m$, and if $G$ is abelian then so is $H_n=G_n$,
\item $\zeta(x;\cbar_n)$ defines a subset $Z_n$ with $\mu_n(Z_n)<\epsilon$,
\item for each $\lambda\in\Lambda$, $\xi_\lambda(x;\dbar_{n,\lambda})$ defines a subset $F_n(\lambda)$ of $H_n$, and $(F_n(\lambda))_{\lambda\in \Lambda}$ forms a partition of $H_n$,
\item if $f_n\colon H_n\to \Lambda$ is defined so that $f_n(x)=\lambda$ if and only if $x\in F_n(\lambda)$, then $f_n$ is a $\delta$-homomorphism from $H_n$ to $\T^r$,
\item $\psi(x;\bbar_n)$ defines the set $Y_n=\{x\in H_n:d(f_n(x),0_{\T^r})<3\delta\}$, and
\item for all $g\in G_n\backslash Z_n$, either $\mu_n(gY_n\cap A_n)< \epsilon^*$ or $\mu_n(gY_n\backslash A_n)< \epsilon^*$.
\end{enumerate}
We claim that $I\in\cU$ by {\L}o\'{s}'s Theorem and since $M\prec G$. In other words, the claim is that conditions $(i)$ through $(vi)$ are first-order expressible (possibly using the expanded language discussed in Remark \ref{rem:Losmu}).   This is clear for $(i)$, $(ii)$, $(iii)$, and $(vi)$, and the only subtleties lie in $(iv)$ and $(v)$. In both cases, the crucial point is that $\Lambda$ is finite, and so these conditions can be described by first-order sentences (using similar ideas as in Proposition \ref{prop:Ydef}). For instance, to express condition $(iv)$, fix $\lambda\in \Lambda$, and let $P_\lambda=\{(\lambda_1,\lambda_2)\in \Lambda^2: d(\lambda,\lambda_1+\lambda_2)<\delta\}$. Let $\sigma_\lambda$ be a sentence expressing that for any $x,y$, if $x\cdot y\in F(\lambda)$ then $x\in F(\lambda_1)$ and $y\in F(\lambda_2)$ for some $(\lambda_1,\lambda_2)\in P(\lambda)$. Then the conjunction  $\bigwedge_{\lambda\in\Lambda}\sigma_\lambda$ expresses precisely that $f$ is a $\delta$-homomorphism.  The details for condition $(v)$ are similar and left to the reader.

Since $I\in\cU$, we may fix $n\in I$ such that $n\geq\max\{m,r,\delta\inv\}$. Since $0<\delta<\alpha$ and $Y_n$ is a $\delta$-approximate $(3\delta,r)$-Bohr neighborhood in $H_n$, it follows from Corollary \ref{cor:findBohr} that $Y_n$ contains a $(\delta,r)$-Bohr neighborhood $B$ in $H_n$. So, by choice of $(G_n,A_n)$, there must be $g\in G_n\backslash Z_n$ such that $\mu_n(gY_n\cap A_n)\geq \gamma(m,r,\delta)\mu_n(B)$ and $\mu_n(gY_n\backslash A_n)\geq \gamma(m,r,\delta)\mu_n(B)$. So, to obtain a contradiction (to $(vi)$), it suffices to show that $\epsilon^*\leq \gamma(m,r,\delta)\mu_n(B)$, i.e.\ (by choice of $\epsilon^*$), show $m\inv\delta^r\leq \mu_n(B)$.  To see this, note that $|B|\geq\delta^r|H_n|$ by Proposition \ref{prop:Bohrgeneric} (applied with $L=\T^r$, so $\ell_{\delta/2}=\delta^r$), and so $|B|\geq \delta^r m\inv|G_n|$ since $[G_n:H_n]=m$.
\end{proof}

We now prove the main result for NIP subsets of arbitrary finite groups.

\begin{theorem}\label{thm:mainNIP}
For any $k\geq 1$ and $\epsilon>0$ there is $n=n(k,\epsilon)$ such that the following holds. Suppose $G$ is a finite group and $A\seq G$ is $k$-NIP. Then there are
\begin{enumerate}[\hspace{5pt}$\ast$]
\item  a normal subgroup $H\leq G$ of index $m\leq n$, 
\item  a $(\delta,r)$-Bohr neighborhood $B$ in $H$, where $r\leq n$ and $\frac{1}{n}\leq\delta\leq 1$, and
\item a subset $Z\seq G$, with $|Z|<\epsilon|G|$,
\end{enumerate}
satisfying the following properties. 
\begin{enumerate}[$(i)$]
\item \textnormal{(structure)} There is a set $D\seq G$, which is a union of at most $m(\frac{2}{\delta})^r$ translates of $B$, such that 
\[
|(A\smd D)\backslash Z|< \epsilon|B|.
\]
\item \textnormal{(regularity)} For any $g\in G\backslash Z$, either $|gB\cap A|<\epsilon|B|$ or $|gB\backslash A|< \epsilon|B|$.
\end{enumerate}
Moreover, $H$ and $Z$ are in the Boolean algebra generated by $\{gAh:g,h\in G\}$, and if $G$ is abelian then we may assume $H=G$.
\end{theorem}
\begin{proof}
Fix $k\geq 1$ and $\epsilon>0$. Define $\gamma\colon(\Z^+)^2\times(0,1]\to\R^+$ such that $\gamma(x,y,z)=\epsilon x\inv(\frac{z}{2})^y$. Let $n=n(k,\epsilon,\gamma)$ be given by Lemma \ref{lem:downstairs}.  Fix a finite group $G$ and a $k$-NIP subset $A\seq G$. By Lemma \ref{lem:downstairs}, there are
\begin{enumerate}[\hspace{10pt}$\ast$]
\item  a normal subgroup $H\leq G$ of index $m\leq n$,
\item  a subset $Y\seq H$, 
\item a $(\delta,r)$-Bohr neighborhood $B$ in $H$, where $r\leq n$ and $\frac{1}{n}\leq\delta\leq 1$, and 
\item a set $Z\subseteq G$, with $|Z|< \epsilon |G|$, 
\end{enumerate}
such that $H,Z$ are in the Boolean algebra generated by $\{gAh:g,h\in G\}$, $B\seq Y\seq G$, and for all $g\in G\backslash Z$, either $|gY\cap A|< \gamma(m,r,\delta)|B|$ or $|gY\backslash A|< \gamma(m,r,\delta)|B|$. Moreover, if $G$ is abelian then we may assume $H=G$. Since $B\seq Y$ and $\gamma(m,r,\delta)\leq \epsilon$, this immediately yields condition $(ii)$. 

For condition $(i)$, we argue as in the proof of Theorem \ref{thm:UPgen}. First, by Proposition \ref{prop:Bohrlike}$(b)$, there is a set $F\seq G\backslash Z$ such that $|F|\leq m(\frac{2}{\delta})^r$ and $G\backslash Z\seq FB$. Let $I=\{g\in F:|gB\backslash A|<\gamma(m,r,\delta)|B|\}$, and note that if $g\in F\backslash I$ then $|gB\cap A|<\gamma(m,r,\delta)|B|$. Let $D=IB$. Since $G\backslash Z\seq FB$, we have
\[
A\smd D\seq Z\cup \bigcup_{g\in I}(gB\backslash A)\cup\bigcup_{g\in F\backslash I}(gB\cap A),\text{~~and so}
\]
\begin{equation*}
|(A\smd D)\backslash Z| \leq \sum_{g\in I}|gB\backslash A|+\sum_{g\in F\backslash I}|g B\cap A|< |F|\gamma(m,r,\delta)|B|\leq \epsilon|B|.\qedhere
\end{equation*}
\end{proof}

We end this section with some remarks on NIP subsets of finite \emph{simple} groups. To motivate this, we first consider the stable case. In particular, given $k\geq 1$ and $\epsilon>0$, it follows immediately from Theorem \ref{thm:CPT1} that if $G$ is a sufficiently large (depending on $k$ and $\epsilon$) finite simple group, and $A\seq G$ is $k$-stable, then $|A|<\epsilon|G|$ or $|A|>(1-\epsilon)|G|$.
On the other hand, this conclusion fails for NIP subsets of abelian finite simple groups (e.g., by the examples of NIP sets in $\Z/p\Z$ mentioned in the introduction). So it is interesting to observe that for \emph{nonabelian} finite simple groups, one recovers the same ``triviality" for NIP sets. 

\begin{corollary}\label{cor:NIPsimple}
Given $k\geq 1$ and $\epsilon>0$, there is $m=m(k,\epsilon)$ such that the following holds. Suppose $G$ is a nonabelian finite simple group, with $|G|>m$, and $A\seq G$ is $k$-NIP. Then $|A|<\epsilon|G|$ or $|A|>(1-\epsilon)|G|$.
\end{corollary}
\begin{proof}
Let $m=n(k,\frac{\epsilon}{2})$ be as in Theorem \ref{thm:mainNIP}. Suppose $G$ is a nonabelian finite simple group, with $|G|>m$, and $A\seq G$ is $k$-NIP. Then we have a normal subgroup $H\leq G$ of index at most $m$, a $(\delta,r)$-Bohr neighborhood $B$ in $H$ (for some $\delta$ and $r$), a set $Z\seq G$ with $|Z|<\frac{\epsilon}{2}|G|$, and a union $D$ of translates of $B$ such that $|(A\smd D)\backslash Z|<\frac{\epsilon}{2}|B|$. So $|A\smd D|\leq |(A\smd D)\backslash Z|+|Z|<\epsilon|G|$. Since $G$ is simple and $|G|>m$, we have $H=G$. Now $B$ contains the kernel $K$ of a homomorphism from $G$ to $\T^r$. So $G/K$ is abelian, which implies $K=B=G$ since $G$ is simple and nonabelian. Now $D$ is either $\emptyset$ or $G$, and the result follows.
\end{proof}

 The previous proof highlights the significance of obtaining Bohr neighborhoods defined from \emph{abelian} compact  Lie groups, which is ultimately possible thanks to Theorem \ref{thm:comm}.

\begin{remark}
Corollary \ref{cor:NIPsimple} is one of the few results on NIP sets in \emph{nonabelian} finite groups for which we know an explicit bound. Specifically, in \cite{CoBogo}, the first author used techniques of Alon, Fox, and Zhao \cite{AFZ} and a result of Gowers \cite{GowQRG} on ``quasirandom" groups to obtain the bound $m=\exp(c(90/\epsilon)^{6k-6})$ in Corollary \ref{cor:NIPsimple}, where $c$ is an absolute constant (see \cite[Remark 8.7]{CoBogo}).  

As for stable sets, the proofs of Corollaries 1 and 3 in \cite{TeWo2} yield an explicit lower bound of the form $|G|\geq \exp(c_k(\epsilon^{\nv d_k}))$ for the analogue of Corollary \ref{cor:NIPsimple} in which $A$ is $k$-stable and $G$ is an \emph{abelian} finite simple group (i.e.,  a cyclic group of prime order). The work in \cite{CoQSAR} removes the exponential and yields improved values for the constants $c_k,d_k$ (see \cite[Theorem 1.3]{CoQSAR}).
\end{remark}

\section{Distal arithmetic regularity}\label{sec:distal}

In this section, we  adapt the preceding results to the case of NIP \emph{fsg} groups with smooth left-invariant measures (e.g., \emph{fsg} groups definable in \emph{distal} theories). In contrast to previous results, where we focused on a single NIP formula, here  we will operate in the setting of groups definable in an NIP theory.

In this section, we let $T$ be a complete theory and we work in a saturated model $M^*$. Given $M\preceq M^*$ and an $M$-definable set $X$ in $M^*$, a \textbf{Keisler measure on $X$ over $M$} is a finitely additive probability measure defined on the Boolean algebra of $M$-definable subsets of $X$. If $M^*=M$ then $\mu$ is a \textbf{global Keisler measure on $X$}. If $T$ is NIP, then a global Keisler measure $\mu$ on a definable set $X$ is \textbf{generically stable} if there is a small model $M\prec M^*$ such that $X$ is $M$-definable, $\mu$ is \emph{$M$-definable} (i.e., for any formula $\phi(x;\ybar)$ with $x$ in the same sort as $X$, the map $\bbar\mapsto \mu(\phi(x;\bbar)\cap X)$ is $M$-definable) and $\mu$ is \emph{finitely satisfiable} in $M$ (i.e., if $Y\seq X$ is definable and $\mu(Y)>0$ then $Y\cap M\neq\emptyset$). See \cite[Section 7.5]{Sibook}) for details.

\begin{definition}
($T$ is NIP.) A definable group $G$ is \emph{\textbf{fsg}} if it admits a generically stable left-invariant (global) Keisler measure.
\end{definition}

We note that this is not the original definition of \emph{fsg} (which is given in \cite{HPP}), but rather the right characterization (in NIP theories) for our purposes (see \cite[Proposition 6.2]{HPP} and \cite[Remark 4.2]{HPS}). The significance of this notion in the context of our work is illustrated by the following example.

\begin{example}\label{ex:pseudofinite}
If $T$ is NIP and $G$ is a definable pseudofinite group, then $G$ is \emph{fsg} since the pseudofinite counting measure on $G$ is generically stable. This follows directly from the VC-Theorem (see, e.g., \cite[Example 7.32]{Sibook}, \cite[Section 2]{CPpfNIP}).
\end{example}

The \emph{fsg} property for a definable group in an NIP theory has strong consequences. For example, we will use the following result from \cite{HP}.

\begin{theorem}[Hrushovski \& Pillay \textnormal{\cite[Theorem 7.7]{HP}}]
An fsg group definable in an NIP theory admits a unique left-invariant Keisler measure $\mu$, which is also the unique right-invariant Keisler measure.
\end{theorem}

In NIP theories, definable \emph{fsg} groups also satisfy  a generic compact domination statement similar to Theorem \ref{thm:G00CP}$(d)$ (see \cite[Corollary 4.9]{SimGCD}). However, in this section, we focus on a certain strengthening of generic compact domination for definable \emph{fsg} groups in NIP theories whose unique left-invariant Keisler measure is  \emph{smooth}. 

\begin{definition} 
Let $X$ be definable in $M^*$. A global Keisler measure $\mu$ on $X$ is \textbf{smooth} if there is a small model $M\prec M^*$ such that $X$ is $M$-definable and $\mu$ is the unique global Keisler measure on $X$ extending $\mu|_{M}$. 
\end{definition}

If $T$ is NIP, then any smooth measure is generically stable (see Proposition 7.10 and Theorem 7.29 in \cite{Sibook}), and so any definable group with a smooth left-invariant Keisler measure is \emph{fsg}. A place to find smooth measures is in the setting of \emph{distal theories}, which were introduced by Simon in \cite{SiDND}. For our purposes we will take the following characterization as a {\em definition} of distality (see \cite[Theorem 1.1]{SiDND}).

\begin{definition} $T$ is \textbf{distal} if it is NIP and every global generically stable Keisler measure is smooth. 
\end{definition}

So, in particular, if $T$ is distal and $G$ is a definable \emph{fsg} group, then the unique left-invariant Keisler measure on $G$ is smooth. In \cite[Proposition 8.41]{Sibook}, Simon shows that if $T$ is countable and NIP, and $G$ is a definable group with a smooth left-invariant Keisler measure, then generic compact domination for $G$ can be strengthened to outright ``compact domination". 
We give a formulation of this result suitable for our applications.  First, recall that if $T$ is NIP and $G$ is a definable group, then $G$ has a smallest type-definable subgroup of bounded index, denoted $G^{00}$, which is an intersection of at most $|T|$ definable sets (see \cite{ShG00} and/or \cite[Proposition 6.1]{HPP}). The reader should compare the following result to Lemma \ref{lem:oneY}.

\begin{lemma} \label{lem:CD} Assume $T$ is countable and NIP, and $G$ is a definable fsg group. Let $\mu$ be the unique left-invariant Keisler measure on $G$, and assume $\mu$ is smooth. Let $(W_i)_{i=0}^\infty$ be a decreasing sequence of definable sets such that $G^{00}=\bigcap_{i=0}^\infty W_i$. Fix a definable set $A\seq G$. Then, for any $\epsilon>0$, there is a definable set $Z\seq G$ and some $i\in\N$ such that $\mu(Z)<\epsilon$ and, for any $g\in G\backslash Z$, either $gW_i\cap A=\emptyset$ or $gW_i\seq A$.
\end{lemma}
\begin{proof}
Let $F=\{C\in G/G^{00}:C\cap A\neq\emptyset\text{ and }C\cap (G\backslash A)\neq\emptyset\}$. Then $F$ is closed by Fact \ref{fact:QLT}$(b)$, and $\eta_{G/G^{00}}(F)=0$ by \cite[Proposition 8.41]{Sibook}. Now fix $\epsilon>0$. By Fact \ref{fact:QLT}$(e)$, there is a definable set $Z\seq G$ such that $\mu(Z)<\epsilon$ and $\{a\in G:aG^{00}\in F\}\seq Z$. By saturation of $G$ and since $(W_i)_{i=0}^\infty$ is decreasing, there is some $i\in\N$ such that for any $g\in G\backslash Z$, either $gW_i\cap A=\emptyset$ or $gW_i\seq A$. 
\end{proof}

We now prove analogues of Theorems \ref{thm:UPprof} and \ref{thm:UPgen} for  an \emph{fsg} group $G$, definable in an NIP theory, such that the unique left-invariant Keisler measure on $G$ is smooth.  In these results,  the assumptions are stronger in the sense that the whole theory is assumed to be NIP. Moreover, in the conclusions we have definability of the data, but no claims about definability in a certain Boolean fragment (see Remark \ref{rem:fsgNIP}). On the other hand, we have outright inclusion or disjointness, rather than up to $\epsilon$, which yields stronger structure and regularity statements.

\begin{theorem}\label{thm:distalprof}
Assume $T$ is NIP. Let $G$ be a definable fsg group, and let $\mu$ be the unique left-invariant Keisler measure on $G$. Suppose $\mu$ is smooth (e.g., if $T$ is distal) and $G/G^{00}$ is profinite. Fix a definable set $A\seq G$ and some $\epsilon > 0$. Then there are
\begin{enumerate}[\hspace{5pt}$\ast$]
\item a definable finite-index normal subgroup $H$ of $G$, and
\item a set $Z\seq G$, which is a union of cosets of $H$ with $\mu(Z)< \epsilon$, 
\end{enumerate}
satisfying the following properties.
\begin{enumerate}[$(i)$]
\item \textnormal{(structure)} $A\backslash Z$ is a union of cosets of $H$.
\item \textnormal{(regularity)} For any $g\in G\backslash Z$, either $gH\cap A=\emptyset$ or $gH\seq A$.
\end{enumerate}
\end{theorem} 
\begin{proof} 
The proof is similar to that of Theorem \ref{thm:UPprof}. First, we may restrict to a countable language in which $G$ and $A$ are definable and $\mu$ is still smooth  (see \cite[Lemma 7.8]{Sibook} and subsequent remarks). By Fact \ref{fact:compactG}$(c)$, we also preserve the hypothesis that $G/G^{00}$ is profinite. 

 By Lemma \ref{lem:Lie}, we may fix a decreasing sequence $(H_i)_{i=0}^\infty$ of definable finite-index normal subgroups of $G$ such that $G^{00}=\bigcap_{i=0}^\infty H_i$. By Lemma \ref{lem:CD}, we have a definable set $Z$ and some $H=H_i$ satisfying condition $(ii)$. As in the proof of Theorem \ref{thm:UPprof}, we may assume $Z$ is a union of cosets of $H$. For condition $(i)$, set $D=\{g\in G:gH\seq A\text{ and }gH\cap Z=\emptyset\}$. Then $D$ is a union of cosets of $H$, and $D\seq A\seq D\cup Z$. Since $D\cap Z=\emptyset$, we have $A\backslash Z=D$.
\end{proof}

Now we prove the general statement.

\begin{theorem}  \label{thm:distalgen}
Assume $T$ is NIP. Let $G$ be a definable fsg group, and let $\mu$ be the unique left-invariant Keisler measure on $G$. Suppose $\mu$ is smooth (e.g., if $T$ is distal). Fix a definable set $A\seq G$ and some $\epsilon > 0$. Then there are
\begin{enumerate}[\hspace{5pt}$\ast$]
\item a definable finite-index normal subgroup $H$ of $G$,
\item a compact connected metric Lie group $(L,d)$, 
\item a definable homomorphism $\pi\colon H\to L$, and 
\item a definable set $Z\seq G$, with $\mu(Z)< \epsilon$, 
\end{enumerate}
such that, for any integer $t\geq 1$, there is
\begin{enumerate}[\hspace{5pt}$\ast$]
\item a definable $(t,\pi)$-approximate Bohr chain 
$(Y_{m})_{m=0}^{\infty}$ in $H$, 
\end{enumerate}
satisfying the following properties, for any $m\in\N$.
\begin{enumerate}[$(i)$]
\item \textnormal{(structure)} There is $D_m\seq G$, which is a union of finitely many left translates of $Y_m$, such that $D_m\seq A\seq D_m\cup Z$. 
\item \textnormal{(regularity)} For any  $g\in G\backslash Z$, either $gY_m\cap A=\emptyset$ or $gY_m\seq A$.
\end{enumerate}
Moreover, if $G$ is pseudofinite then we may assume $L=\T^n$ for some $n\in\N$; and if $G/G^{00}$ is abelian then we may assume $L=\T^n$ for some $n\in\N$ and $H=G$.
\end{theorem}
\begin{proof}
The proof is similar to that of Theorem \ref{thm:UPgen}. First, as in the proof of Theorem \ref{thm:distalprof}, we may assume $T$ is in a countable language. Let $(W_i)_{i=0}^\infty$ be a decreasing sequence of definable subsets of $G$ satisfying the conditions of Proposition \ref{prop:Bohr}, with $\Gamma=G^{00}$. By Lemma \ref{lem:CD}, there is a definable set $Z\seq G$ and some $i\in\N$ such that $\mu(Z)<\epsilon$ and, if $W:=W_i$, then for all $g\in G\backslash Z$, either $gW\cap A=\emptyset$ or $gW\seq A$. Proposition \ref{prop:Bohr} associates to $W$ a definable homomorphism $\pi\colon H\to L$, with $\ker\pi\seq W$, where $H$ is a definable finite-index normal subgroup of $G$ and $(L,d)$ is compact connected metric Lie group. Note also that the ``moreover" statement is given by Proposition \ref{prop:Bohr}. By Lemma \ref{lem:Bohrapprox}, there is a definable $(t,\pi)$-approximate Bohr chain $(Y_m)_{m=0}^\infty$ in $H$. By saturation,  if $m$ is sufficiently large then $Y_m\seq W$ and so if $g\in G\backslash Z$ then either $gY_m\cap A=\emptyset$ or $gY_m\seq A$. This yields condition $(ii)$. For condition $(i)$, mimic the end of the proof of Theorem \ref{thm:UPgen} to find $D_m\seq G$, which is a union of finitely many left translates of $Y_m$, such that $D_m\seq A\seq D_m\cup Z$ (replace each occurrence of $\mu(A\cap gY_m)=0$ with $A\cap gY_m=\emptyset$, and each occurrence of $\mu(A\backslash gY_m)=0$ with $gY_m\seq A$).
\end{proof}

Since Theorems \ref{thm:distalprof} and \ref{thm:distalgen} use global assumptions on the theory, a fully general statement for applications to finite groups would be rather cumbersome to state. However, one concludes, in a routine fashion, structure and regularity for suitable families of finite groups, for example a collection $\cG$ of finite $\cL$-structures expanding groups, such that any completion of $\Th(\cG)$ is distal. In the next section, we will discuss an application to the family of finite groups obtained as quotients of a compact $p$-adic Lie group by its open normal subgroups.

\begin{remark}\label{rem:fsgNIP}
Results along the lines of Theorems \ref{thm:UPprof} and \ref{thm:UPgen} can also be shown for \emph{fsg} groups definable in NIP theories, but without the smoothness assumption.  Indeed, such groups satisfy generic compact domination just as in Theorem \ref{thm:G00CP}$(d)$, but with $G^{00}_{\theta^r}$ replaced by $G^{00}$ (see \cite[Corollary 4.9]{SimGCD}). However, as discussed in \cite[Remark 1.3]{CPpfNIP}, since the unique left-invariant measure on an NIP \emph{fsg} group $G$ is generically stable, one could fix an invariant formula $\theta(x;\ybar)$ and construct $G^{00}_{\theta^r}$ as in Theorem \ref{thm:G00CP}. This yields structure and regularity theorems as before with additional information about the definability of the data. Precisely: 

Assume $T$ is NIP. Suppose $G$ is definable and \emph{fsg}, and let $\mu$ be the unique left-invariant Keisler measure on $G$. Without loss of generality, assume $G=M^*$. Fix an invariant formula $\theta(x;\ybar)$. Then, for any $\theta^r$-definable $A\seq G$ and any $\epsilon>0$, we have the conclusion of Theorem \ref{thm:UPgen}, except with $\T^n$ replaced by some compact connected metric Lie group $(L,d)$. If $G/G^{00}_{\theta^r}$ is profinite, then the conclusion of Theorem \ref{thm:UPprof} holds exactly as stated. 

It would be interesting to pursue notions of smoothness for local measures, or ``local distality" for formulas, and recover local versions of Theorems \ref{thm:distalgen} and \ref{thm:distalprof}.
For instance, one might consider an NIP formula $\theta(x,\ybar)$ such that every generically stable global Keisler measure on the Boolean algebra of $\theta$-formulas is smooth.
\end{remark}

\section{Compact $p$-adic analytic groups}\label{sec:padic}

In this section, we apply Theorem \ref{thm:distalprof} to the setting of compact $p$-adic analytic groups (see Theorems \ref{thm:padic1} and \ref{thm:padic2}).  
We assume some familiarity with the $p$-adic field $\Q_{p}$ and $p$-adic model theory. See \cite{Bel-p-adic}  for further reading. The topology on $\Q_{p}$  is given by the valuation $v$ where open neighborhoods of a point $a$ are defined by $v(x-a)\geq n$ for $n\in \Z$.  The topology on $\Q_{p}^{n}$ is the product topology.  A \textbf{$p$-adic analytic function} is a function $f$, from some open $V\subseteq \Q_{p}^{n}$ to $\Q_{p}$, such that for every $a\in V$, there is an open neighborhood of $a\in V$ in which $f$ is given by a convergent power series. We obtain the notions of a $p$-adic analytic manifold and a $p$-adic analytic (or Lie) group. Recall that any compact $p$-adic analytic group is profinite (see also \cite[Section 2.12.1]{RZbook}).

We let $\Q_{p}^{\an}$ denote the expansion of the field $(\Q_{p},+,\cdot)$ by symbols for all convergent (in $\Z_p$)  power series in $\Z_{p}[[X_{1}, \ldots, X_{n}]]$ for all $n$.  Then any compact $p$-adic analytic manifold or group is seen to be naturally definable in the structure $\Q_{p}^{\an}$ (we conflate definable and interpretable at this point). For example, given a compact $p$-adic analytic group, we can find an atlas consisting of finitely many open definable (even semialgebraic) sets with analytic transition functions and such that group operation is analytic when read in the charts.
 When we talk about a compact $p$-adic analytic group being definable in $\Q_{p}^{\an}$ we will mean definable in such a manner.

  We note that $\Th(\Q_p^{\an})$ is distal. Indeed, this follows from several results in the literature: distality of $\Th(\Q_{p}, +,\cdot)$ \cite[Example 9.20]{Sibook}, \emph{dp}-minimality of $\Th(\Q_{p}^{\an})$ (see Corollary 7.9 and Remark 7.10 of \cite{ADHMS1}), and the fact that a \emph{dp}-minimal expansion of a distal and \emph{dp}-minimal theory is distal \cite[Remark 6.7]{ChSt}.  It is also well known that distality passes from $T$ to $T^{\textnormal{eq}}$ (see \cite[Exercise 9.12]{Sibook}).

The next lemma  follows from \cite{MacTePAL}, and we will give an explanation afterwards. 

\begin{lemma}\label{lem:unifdef} 
Let $K$ be a compact $p$-adic analytic group. Then the family of open normal subgroups is uniformly definable in $\Q^{\an}_p$.
\end{lemma}

\noindent 
{\em Explanation.} First, by Theorem 2.1 of \cite{MacTePAL}, $K$ has a open normal subgroup $G$  (so of finite index) which is a ``uniformly powerful pro-$p$ group of finite (topological) rank  $d$"  (see \cite{MacTePAL} for the definitions).
Proposition 1.2 of \cite{MacTePAL}, and its proof, states that $G$ is isomorphic (as a topological group) to 
$\Z_{p}^{d}$ equipped with a certain analytic group structure $\ast$ (and note  $(\Z_{p}^{d},\ast)$ is definable in $\Q_{p}^{\an}$), such that moreover the collection of open subgroups of $(\Z_{p}^{d}, \ast)$ is (uniformly) definable in $\Q_{p}^{\an}$.  By Theorem 10.5 of \cite{DdSMS}, there is an analytic (so definable in $\Q_{p}^{\an}$)  isomorphism between $G$ and $(\Z_{p}^{d},\ast)$, whereby the family of open subgroups of $G$ is uniformly definable in $\Q_{p}^{\an}$.  As $G$  has finite index and is normal in $K$, it follows that the collection of open subgroups of $K$ is uniformly definable in $\Q_{p}^{\an}$.  (See the last part of the proof of \cite[Theorem 1.1$((1)\Rightarrow(2))$]{MacTePAL} on page 1043 of that paper.) \medskip

In particular, the collection of quotients of a compact $p$-adic analytic group by its open normal subgroups is a family of uniformly definable (in $\Q^{\an}_p$) finite groups. Let us fix a compact $p$-adic analytic group $K$ (so definable in $\Q_{p}^{\an}$). If $M^*\succ \Q_p^{\an}$ is sufficiently saturated, then  $K(M^*)$ denotes the group definable in the structure $M^*$ by the same formula as the one defining $K$ in $\Q_{p}^{\an}$. The next result combines the ``Claim" in \cite[Section 6]{HPP} with Corollaries 2.3 and 2.4 of \cite{OnPi}.

\begin{fact}\label{fact:Qpan} \textnormal{\cite{HPP,OnPi}} Let $M^*\succ\Q_{p}^{\an}$  be sufficiently saturated. Then $K(M^*)$ is fsg. Moreover, $K(M^*)/K(M^*)^{00}$ is isomorphic to $K$, and thus is profinite. In particular, $K(M^*)^{00}=\bigcap_{i=0}^\infty H_i(M^*)$ where $\{H_i:i\in\N\}$ is the neighborhood basis at the identity consisting of open normal subgroups of $K$.
\end{fact}

Our first application of Theorem \ref{thm:distalprof} is the following ``structure and regularity" statement for definable subsets of compact $p$-adic analytic groups. In fact, as we point out below,  this can also be seen as a fairly direct application of \cite[Proposition 2.8]{OnPi}, and an extension of certain results in that paper.

\begin{theorem} \label{thm:padic1}
Let  $K$ be a compact $p$-adic analytic group (so definable in the structure $\Q_{p}^{\an}$). Let $A\subseteq K$ be definable in $\Q_{p}^{\an}$, and let
 $\epsilon > 0$. Then there are
\begin{enumerate}[\hspace{5pt}$\ast$]
\item an open (so finite-index) normal subgroup $H$ of $K$, and
\item a set $Z\subseteq K$, which is a union of cosets of $H$ with $\eta_K(Z) < \epsilon$, 
\end{enumerate}
satisfying the following properties.
\begin{enumerate}[$(i)$]
\item \textnormal{(structure)} $A\backslash Z$ is a union of cosets of $H$.
\item \textnormal{(regularity)}  For any $g\in K\backslash Z$, either $gH\cap A = \emptyset$ or $gH\subseteq A$.
\end{enumerate}
\end{theorem}
\begin{proof}
Let $M^*\succ \Q_p^{\an}$ be sufficiently saturated. By distality, we may apply Theorem \ref{thm:distalprof} to the group $K(M^*)$ and the definable set $A(M^*)$. Since $K(M^*)^{00}$ is contained in any definable finite-index subgroup of $K(M^*)$, we may use Fact \ref{fact:Qpan} to further assume that the finite-index normal subgroup of $K(M^*)$ given by Theorem \ref{thm:distalprof} is of the form $H(M^*)$ for some open normal subgroup $H$ of $K$. So the error set is of the form $Z(M^*)$, where $Z$ is definable in $\Q_p^{\an}$. Note also that if $\mu$ is the (unique) left-invariant Keisler measure on $K(M^*)$, then $\mu(H(M^*))=\eta_K(H)=1/[K:H]$ since $H$ is a finite index subgroup of $K$, and thus $\mu(Z(M^*))=\eta_K(Z)$.  Altogether, by elementarity, we have $(i)$ and $(ii)$.
\end{proof}

\begin{remark} $~$
\begin{enumerate}
\item The proof of Theorem \ref{thm:padic1} only requires  definability of the open normal subgroups of $K$ (and not \emph{uniform} definability).
\item Proposition 2.8 of \cite{OnPi} states that $K(M^*)$ is compactly dominated via the map $K(M^*)\to K(M^*)/K(M^*)^{00}$. So we could also have deduced Theorem \ref{thm:padic1} from this result, together with the standard methods.
\item It would be interesting to prove Theorem \ref{thm:padic1} from the cell decomposition results of Denef and others (at least when $K=\Z_{p}^{n}$ for some $n$).  
\end{enumerate}
\end{remark} 

Our second application of Theorem \ref{thm:distalprof} is to the family of quotients of a compact $p$-adic analytic group by its open normal subgroups.

\begin{theorem} \label{thm:padic2}
Let $K$ be a compact $p$-adic analytic group. Let $(G_{i})_{i\in I}$ be the family of finite groups obtained as quotients of $K$ by open normal subgroups. Let $A\seq K$ be definable in $\Q_{p}^{\an}$ and for $i\in I$, let $A_{i}\seq G_{i}$ be the image of $A$ under the quotient map.  Fix $\epsilon>0$.  There is some $n=n(K,A,\epsilon)$ such that for any $i\in I$, there are
\begin{enumerate}[\hspace{5pt}$\ast$]
\item a normal subgroup $H_{i}\leq G_{i}$ of index at most $n$, and
\item a set $Z_{i}\seq G_i$, which is a union of cosets of $H_{i}$ with $|Z_{i}|<\epsilon |G_{i}|$,
\end{enumerate}
satisfying the following properties.
\begin{enumerate}[$(i)$]
\item \textnormal{(structure)} $A_{i}\backslash Z_{i}$ is a union of cosets of $H_{i}$.
\item \textnormal{(regularity)} For any $g\in G_{i}\backslash Z_{i}$, either $gH_{i}\cap A_{i} = \emptyset$ or $gH_{i}\subseteq A_{i}$.
\end{enumerate}
\end{theorem} 
\begin{proof}
It suffices to prove condition $(i)$.  Toward a contradiction, suppose we have $\epsilon>0$ such that, for any $n\geq 1$ there is some $i_n\in I$ such that if $H\leq G_{i_n}$ is a normal subgroup of index at most $n$, and $Z\seq G_{i_n}$ is a union of cosets of $H$ with $|Z|<\epsilon|G_{i_n}|$, then $A_{i_n}\backslash Z$ is not a union of cosets of $H$. 

Let $G_i=K/U_i$, where $(U_i)_{i\in I}$ lists the open normal subgroups of $K$. By Lemma \ref{lem:unifdef}, we may view $U_t$ as a formula (over $\Q_p$) in the variable $t$, and $I$ as a definable set in $\Q_p^{\an}$ in the sort for $t$. Let $S_I(\Q_p)$ be the space of types with parameters from $\Q_p$ concentrating on $I$. Let $M^*\succ\Q^{\an}_p$ be sufficiently saturated, and let $i_*\in I(M^*)$ realize an accumulation point of $\{\tp(i_n/\Q_p):n\geq 1\}$ in $S_I(\Q_p)$. Let $G=K(M^*)/U_{i_*}$. Then $G$ is a pseudofinite group definable in $M^*$, and thus is \emph{fsg} (see Example \ref{ex:pseudofinite}). Since $G$ is a definable quotient of $K(M^*)$, and $K(M^*)/K(M^*)^{00}$ is profinite by Fact \ref{fact:Qpan}, it follows (using Fact \ref{fact:compactG}$(c)$) that $G/G^{00}$ is also profinite.  By Theorem \ref{thm:distalprof}, there are $m,n,r\in\N$, with $n\geq 1$, $m<\epsilon n$, and $r\leq n$, and a formula $\phi(x;\ybar)$ (with no parameters) such that:

$(\ast)$ For some $\bbar\in (M^*)^{\ybar}$, $H:=\phi(M^*;\bbar)$ is a  normal subgroup of $G=K(M^*)/U_{i^*}$ of index $n$, and there is $Z\seq G$, which is a union of $m$ cosets of $H$ such that $(A(M^*)/U_{i^*})\backslash Z$ is a union of $r$ cosets of $H$. 

In particular, $(\ast)$ is a property of $i^*$ that can be expressed using a formula $\zeta(t)$  over $\Q_p$. So there is $n'\geq n$ such that $\Q^{\an}_p\models\zeta(i_{n'})$, which is a contradiction. 
\end{proof}

\begin{remark}
Suppose $K$ is a compact $p$-adic analytic group such that the family of open normal subgroups is (eventually) linearly ordered by inclusion (e.g., $K=\Z_p$). In this case, Theorem \ref{thm:padic2} can be deduced easily from Theorem \ref{thm:padic1}. Indeed, given a definable set $A\seq G$ and some $\epsilon>0$, let $H\leq K$ and $Z\seq K$ be as in Theorem \ref{thm:padic1}. Then all but finitely many open normal subgroups of $K$ are contained in $H$ and, given such a subgroup $U_i\leq H$, if $H_i=H/U_i$ and $Z_i=Z/U_i$, then $H_i$ and $Z_i$ satisfy conditions $(i)$ and $(ii)$ in Theorem \ref{thm:padic2}. Choosing $n$ sufficiently large, we can then let $H_i$ be trivial for any $U_i$ not contained in $H$.

It is not clear whether such an argument can be given for arbitrary $K$ (or if there is a proof of Theorem \ref{thm:padic2} not relying on Lemma \ref{lem:unifdef}). Note that the above assumption does not hold for every compact $p$-adic analytic group; for example $(\Z_p^2,+)$ does not have eventually linearly ordered open (normal) subgroups.
\end{remark}

\bibliographystyle{amsplain}
\end{document}